\pdfoutput=1
\RequirePackage{ifpdf}
\ifpdf % We~are running pdfTeX in pdf mode
\documentclass[pdftex]{sigma}
\else
\documentclass{sigma}
\fi

\numberwithin{equation}{section}

\newtheorem{Theorem}{Theorem}[section]
\newtheorem*{Theorem*}{Theorem}

\newtheorem{Proposition}[Theorem]{Proposition}
 { \theoremstyle{definition}
\newtheorem{Definition}[Theorem]{Definition}

\newtheorem{Example}[Theorem]{Example}
\newtheorem{Remark}[Theorem]{Remark} }

\def\bending{\kappa}
\def\R{\mathbb{R}}
\def\curvature{{\rm k}}
\def\SL{{\rm SL}(2,\R)}
\def\PSL{{\rm PSL}(2,\mathbb{C})}

\def\AdS{{\rm AdS}}
\def\sech{{\rm sech}}
\def\csch{{\rm csch}}

\begin{document}
\allowdisplaybreaks

\newcommand{\arXivNumber}{2312.10765}

\renewcommand{\PaperNumber}{009}

\FirstPageHeading

\ShortArticleName{Geometric Transformations on Null Curves in the Anti-de~Sitter 3-Space}

\ArticleName{Geometric Transformations on Null Curves\\ in the Anti-de~Sitter 3-Space}

\Author{Emilio MUSSO~$^{\rm a}$ and \'Alvaro P\'AMPANO~$^{\rm b}$}

\AuthorNameForHeading{E.~Musso and A.~P\'ampano}

\Address{$^{\rm a)}$~Dipartimento di Scienze Matematiche, Politecnico di Torino,\\
\hphantom{$^{\rm a)}$}~Corso Duca degli Abruzzi 24, I-10129 Torino, Italy}
\EmailD{\href{mailto:emilio.musso@polito.it}{emilio.musso@polito.it}}

\Address{$^{\rm b)}$~Department of Mathematics and Statistics, Texas Tech University, Lubbock, TX, 79409, USA}
\EmailD{\href{mailto:alvaro.pampano@ttu.edu}{alvaro.pampano@ttu.edu}}

\ArticleDates{Received September 24, 2024, in final form February 05, 2025; Published online February 12, 2025}

\Abstract{We provide a geometric transformation on null curves in the anti-de~Sitter $3$-space (AdS) which induces the B\"{a}cklund transformation for the KdV equation. In~addition, we show that this geometric transformation satisfies a suitable permutability theorem. We~also illustrate how to implement it when the original null curve has constant bending.}

\Keywords{anti-de~Sitter space; geometric transformations; KdV equation; null curves}

\Classification{37K35; 53B30}

\section{Introduction}

In a previous paper~\cite{MP}, we showed that, in the context of Lorentzian geometry on the anti-de~Sitter $3$-space ($\AdS$, for short) there are flows (referred to as LIEN flows) on null curves that induce bending evolution by any partial differential equation (PDE) in the KdV hierarchy. In~particular, the first order LIEN flow induces evolution by the Korteweg--de Vries (KdV) equation in the form
\begin{equation*}
\partial_t\bending+\partial_s^3\bending-6\bending \partial_s\bending=0.
\end{equation*}
Although similar flows exist on other Lorentzian space forms (see, for instance,~\cite{AGL,MN2}), the specific case of $\AdS$ is special in the sense that the theory of null curves in $\AdS$ is closely related to the theory of star-shaped curves in the centro-affine plane. Indeed, null curves in $\AdS$ can be analyzed by combining together pairs of star-shaped curves in the centro-affine plane~\cite{MP}.

The KdV equation is a classical PDE that originated as a model to understand the propagation of waves on shallow water surfaces~\cite{KV} and is a prototype of completely integrable evolution equations. As a completely integrable PDE the KdV equation has a rich structure including, not only infinitely many conservation laws, but also a B\"{a}cklund transformation~\cite{TU} which generates new solutions from old. In this paper, we use Tabachnikov's transformation on star-shaped curves~\cite{T} and a similar transformation of Terng--Wu~\cite{TW1} to define a geometric transformation (referred to as the \emph{$\mathcal{T}$-transform}) on null curves in $\AdS$ that corresponds to the B\"{a}cklund transformation for their bendings.

We begin, in Section~\ref{section2}, by recalling the basic properties of null curves $\gamma\colon J\subseteq\R\longrightarrow\AdS$ without inflection points in the $\SL$ model for $\AdS$. Such curves possess canonical parameterizations and a third order differential invariant $\bending$, the \emph{bending} (often known as curvature or torsion). In addition, they also have a \emph{spinor frame field} $(F_+,F_-)\colon J\subseteq\R\longrightarrow\SL\times\SL$ defined along them. The components $F_\pm$ of the spinor frame field along $\gamma$ are, precisely, the~canonical central affine frame fields of two star-shaped curves $\eta_\pm$ in the centro-affine plane $\dot{\R}^2=\R^2\setminus\{(0,0)\}$ with central affine curvatures $\curvature_\pm=\bending\pm 1$, respectively~\cite{MP}.

In Section~\ref{section3}, we introduce the $\mathcal{T}$-transform on null curves in $\AdS$. We show, in Theorem~\ref{T2}, how to build such a transformation using a solution of the Riccati equation
\begin{equation*}
	f'+f^2=\bending+\cosh(2\xi),
\end{equation*}
involving the bending $\bending$ of the null curve and a real parameter $\xi\neq 0$. We also prove that the $\mathcal{T}$-transform satisfies a suitable permutability theorem (Theorem~\ref{T3}).

Next, in Section~\ref{section4}, after briefly recalling some basic facts about the B\"{a}cklund transformation for the KdV equation~\cite{TU,TW2}, we use the $\mathcal{T}$-transform to find, in Theorem~\ref{T4}, a geometric realization of it. In addition, such a transformation satisfies a permutability theorem as well. This geometric realization consists on applying the $\mathcal{T}$-transform to a solution of the LIEN flow~\cite{MP}.

Finally, in Section~\ref{section5}, we illustrate how to implement the $\mathcal{T}$-transform when the original null curve has constant bending. A detailed analysis of null curves in $\AdS$ with constant bending has been carried out in~\cite{MP}.

The results of this paper naturally suggest further directions for research. In particular, it is likely that these results (as well as those of~\cite{MP}) can be reformulated in the context of holomorphic null curves in $\PSL$. Since the seminal work of R.~Bryant~\cite{Br1}, it is known that the geometry of holomorphic null curves \cite{Br2,MN3} in $\PSL$ is related via the Bryant's correspondence to the geometry of immersed surfaces with constant mean curvature one (CMC~$1$, for short) in the hyperbolic space $\mathbb{H}^3$. In this case, it is expected that the $\mathcal{T}$-transform corresponds to the~Darboux transform for CMC~$1$ surfaces in $\mathbb{H}^3$. It is also reasonable to ask if the stationary solutions of the holomorphic LIEN flow or its finite-gap solutions can be used to find new embedded CMC~$1$ surfaces with a finite number of ends, in the spirit of~\cite{BPS}.

\section{Preliminaries}\label{section2}

In this section, we will collect the basic information about the $\SL$ model for the anti-de~Sitter $3$-space ($\AdS$) and about the geometric features of null curves in $\AdS$. For a more detailed analysis, we refer the reader to~\cite[Sections 2 and 3]{MP}.

\subsection{The anti-de~Sitter 3-space}

Consider the vector space of $2\times 2$ real matrices $\R^{2,2}$ equipped with the quadratic form ${\rm q}$ of signature $(-,-,+,+)$ defined by
\begin{equation*}{\rm q}(X)=-\det (X)=x_1^2x_2^1-x_1^1x_2^2,\end{equation*}
for each $X=\bigl(x_i^j\bigr)\in\R^{2,2}$. The corresponding inner product will be denoted by $\langle\cdot,\cdot\rangle$ and we will consider the orientation of $\R^{2,2}$ determined by the volume form $\Omega={\rm d} x_1^1\wedge {\rm d} x_2^2\wedge {\rm d} x_1^2\wedge {\rm d} x_2^1$. On~$\Lambda^2\bigl(\R^{2,2}\bigr)$ we define an inner product $\langle\langle\cdot,\cdot\rangle\rangle$ by
\begin{equation*}\langle\langle U\wedge V,W\wedge Z\rangle\rangle=\det \begin{pmatrix} \langle U,W\rangle & \langle U,Z\rangle \\ \langle V,W\rangle & \langle V,Z\rangle \end{pmatrix}.\end{equation*}
A bivector $U\wedge V\in\Lambda^2\bigl(\R^{2,2}\bigr)$ is of type $(-,-)$ if the restriction of the above inner product to ${\rm span}\{U\wedge V\}$ is negative definite, and of type $(-,0)$ if this restriction is nonzero semi-negative definite and degenerate. We fix the bivector of type $(-,-)$
\begin{equation*}U\wedge V=\begin{pmatrix} 1 & 0 \\ 0 & 1 \end{pmatrix}\wedge\begin{pmatrix} 0 & 1 \\ -1 & 0 \end{pmatrix}\in\Lambda^2\bigl(\R^{2,2}\bigr),\end{equation*}
and we say that a bivector $W\wedge Z$ of type $(-,0)$ is \emph{positive} if $\langle\langle U\wedge V,W\wedge Z\rangle\rangle>0$. The set of all positive bivectors of type $(-,0)$ is denoted by $\mathcal{N}_+$.
This choice defines the notion of time orientation in $\R^{2,2}$.

Let $\SL$ be the \emph{special linear group} of degree $2$ over $\R$ (i.e.,
the group consisting of the~${2\times 2}$ real matrices of determinant $1$ with the ordinary matrix multiplication). The restriction of the inner product $\langle\cdot,\cdot\rangle$ of $\R^{2,2}$ to the special linear group $\SL$ gives a Lorentzian metric of constant sectional curvature $-1$. The special linear group $\SL$ endowed with the Lorentzian metric induced by $\langle\cdot,\cdot\rangle$ is a model for the \emph{anti-de~Sitter $3$-space} ($\AdS$).

\begin{Remark}%\label{H}
Throughout this paper, it will be implicitly assumed that the model for $\AdS$ is the special linear group $\SL$. For visualization purposes, we will identify $\AdS$ with the open solid torus in $\mathbb{R}^3$ swept out by the rotation of the (open) unit disc of the $Oxz$-plane centered at $(2,0,0)$ around the $Oz$-axis.
\end{Remark}

Consider the normal vector field $Q\colon X\in\AdS\longmapsto-2X\in\R^{2,2}$ and the orientation in $\AdS$ determined by the volume form $i_Q\Omega$. Given a null tangent vector $T\in T_X\AdS$, the bivector~${X\wedge T}$ is of type $(-,0)$. We define a time-orientation on $\AdS$ by declaring $T$ to be future-directed if~${X\wedge T\in\mathcal{N}_+}$.

\subsection{Geometry of null curves}

Let $J\subseteq\R$ be an open interval. A smooth immersed curve $\gamma\colon J\subseteq \mathbb{R}\longrightarrow \AdS$ is \emph{null} if its velocity vector $\gamma'(s)$ is a null (or, light-like) vector for each $s\in J$. In other words, if $\langle \gamma'(s),\gamma'(s)\rangle=0$ for all $s\in J$. A null curve is \emph{future-directed} if the bivector $\gamma\wedge\gamma'$ of type $(-,0)$ is positive, that~is, if $\gamma\wedge\gamma'\in\mathcal{N}_+$.

Let $\gamma\colon J\subseteq\R\longrightarrow\AdS$ be a future-directed null curve without inflection points (i.e., such that $\gamma'\wedge\gamma''\neq 0$ holds). Since $\gamma'\wedge\gamma''\neq 0$, then $\gamma''$ is a space-like vector. Using the terminology of~\cite{MP}, we say that $\gamma$ is parameterized by the \emph{proper time} if $\langle \gamma''(s),\gamma''(s)\rangle=4$ for every $s\in J$. The \emph{bending} of a curve $\gamma$ parameterized by the proper time is defined by
\begin{equation*}\bending(s)=-\frac{1}{16}\bigl\langle\gamma'''(s),\gamma'''(s)\bigr\rangle.\end{equation*}

\begin{Remark}\label{assumptions} From now on, we assume that all our curves $\gamma\colon J\subseteq\R\longrightarrow\AdS$ are null, future-directed, parameterized by their proper time $s\in J$, and have no inflection points. In this setting, the bending $\bending(s)$ completely determines the null curve $\gamma$, up to congruence.
\end{Remark}

Let $\gamma\colon J\subseteq\R\longrightarrow\AdS$ be a null curve with bending $\bending$. The \emph{spinor frame field} along $\gamma$ is a~lift $(F_+,F_-)\colon J\subseteq\R\longrightarrow\SL\times\SL$ such that
\begin{gather}
	{\rm d}F_+ = F_+ \begin{pmatrix} 0 & \bending+1 \\ 1 & 0 \end{pmatrix}{\rm d} s,\label{dF+}\\
	{\rm d}F_- = F_- \begin{pmatrix} 0 & \bending-1 \\ 1 & 0 \end{pmatrix}{\rm d} s.\label{dF-}
\end{gather}
The spinor frame field is defined up to a sign. The equations \eqref{dF+} and \eqref{dF-} are the spinorial counterpart of the classical Frenet-type equations and, hence, we will refer to them as the~\emph{spinorial Frenet-type equations} of $\gamma$. Observe that the null curve $\gamma$ can be recovered from the~spinor frame field $(F_+,F_-)$ as
$\gamma=F_+F_-^{-1}$.

\begin{Remark} The frame $(F_+,F_-)$ can be explicitly obtained as a lifting to $\SL\times\SL$ of a moving frame adapted to the curve $\gamma$, with values in the automorphism group of $\AdS$. For the details of the construction, we refer to~\cite{MP}.
\end{Remark}

In~\cite{MP}, employing the spinor frame field $(F_+,F_-)$ along $\gamma$, we related null curves in $\AdS$ to a~suitable pair of star-shaped curves in the centro-affine plane $\dot{\R}^2=\R^2\setminus\{(0,0)\}$. For later use, we state this result here.

\begin{Theorem}[{\cite[Theorem 3.3]{MP}}]\label{relation} Let $\gamma\colon J\subseteq\R\longrightarrow\AdS$ be a null curve with bending $\bending$ and spinor frame field $(F_+,F_-)$ along it $($which can be obtained by solving the spinorial Frenet-type equations \eqref{dF+} and \eqref{dF-}$)$. Then, the first column vectors of $F_\pm$ form a pair of star-shaped curves ${(\eta_+,\eta_-)}$ in $\dot{\R}^2$ with central affine curvatures\footnote{Our notion of central affine curvature coincides with that of Terng--Wu~\cite{TW1} and it has the opposite sign of that of Pinkall~\cite{P}.}
$\curvature_+=\bending+1$, $\curvature_-=\bending-1$, respectively.
	
	Conversely, let $(\eta_+,\eta_-)$ be a pair of star-shaped curves with central affine curvatures~$\curvature_+$ and~$\curvature_-$ related by $(\curvature_+-\curvature_-)/2=1$, and canonical central affine frame fields $F_+=\bigl(\eta_+,\eta_+'\bigr)$ and~${F_-=\bigl(\eta_-,\eta_-'\bigr)}$, respectively. Then, $\gamma=F_+F_-^{-1}$,
	is a null curve in $\AdS$ with bending $\bending=(\curvature_++\curvature_-)/2$ and spinor frame field $(F_+,F_-)$ along it.
\end{Theorem}

\begin{Remark} Following the terminology introduced in~\cite{MP}, the pair of star-shaped curves ${(\eta_+,\eta_-)}$ will be referred to as the pair of cousins associated with the null curve $\gamma$.
\end{Remark}

We briefly recall here the case of closed null curves with constant bending (for more details, see~\cite[Example 3.5]{MP}).

\begin{Example}\label{example} Consider a closed null curve $\gamma\colon J=\R\longrightarrow\AdS$ with constant bending $\bending$. Since~$\gamma$ is closed, this implies that the spinor frame field $(F_+,F_-)$ along $\gamma$ is periodic and that the least periods of $F_+$ and $F_-$ are commensurable. That is, the bending of $\gamma$ must be
	\begin{equation*}\bending\equiv\bending_{m,n}=-\frac{m^2+n^2}{m^2-n^2},\end{equation*}
	where $m>n$ are two relatively prime natural numbers. Furthermore,
	\begin{enumerate}\itemsep=0pt
		\item If $m+n$ is even, the curve $\gamma$ represents a torus knot of type $((m-n)/2,(n+m)/2)$.
		\item If $m+n$ is odd, the curve $\gamma$ represents a torus knot of type $(m-n,n+m)$.
	\end{enumerate}
	In Figure~\ref{constant} we illustrate a closed null curve with constant bending as well as its associated pair of star-shaped cousins (other illustrations of closed null curves of this family can be found in~\mbox{\cite[Figures~2 and~3]{MP}}).
\end{Example}

\begin{figure}[h]\centering
			\includegraphics[height=4cm,width=4cm]{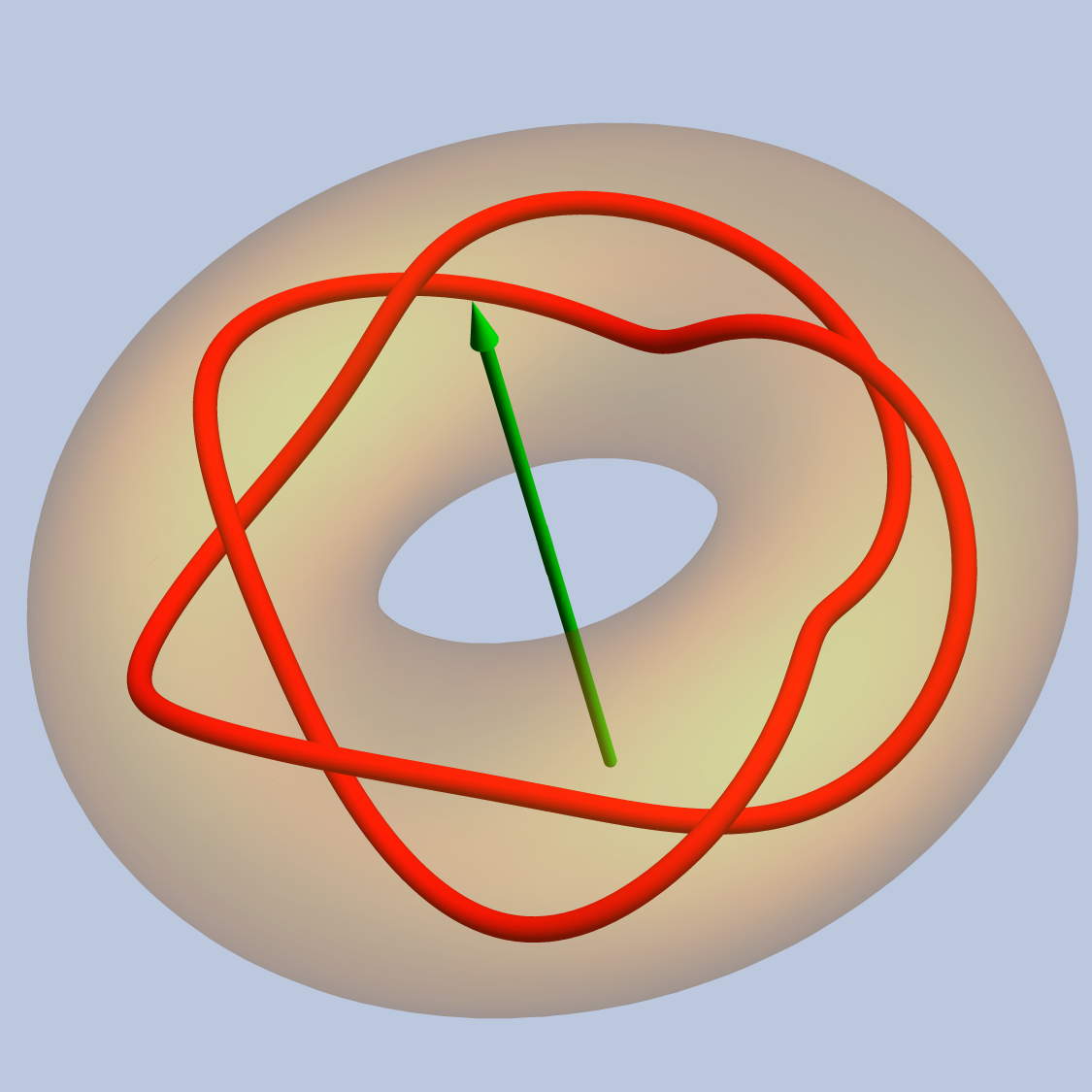}\quad\quad\quad
			\includegraphics[height=4cm,width=4cm]{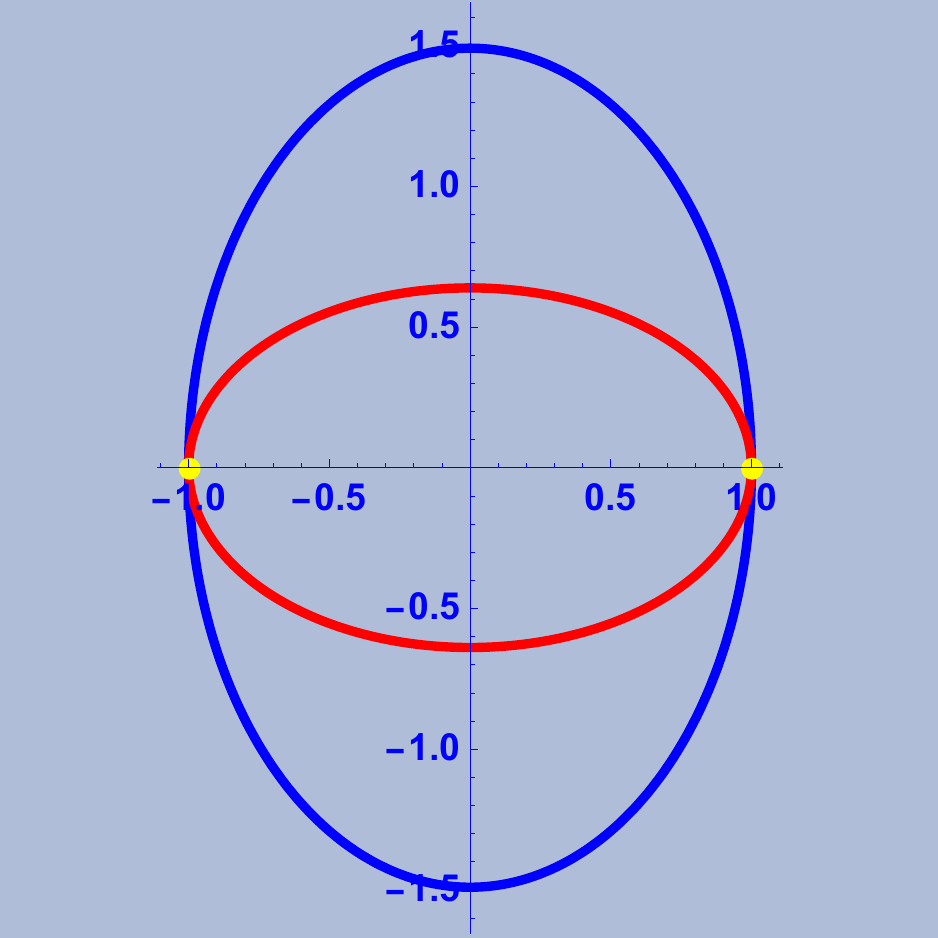}
\caption{Left: The closed null curve $\gamma$ with constant bending $\bending_{7,3}$ (see Example~\ref{example}). This curve represents a torus knot of type $(2,5)$. Right: The associated pair of star-shaped cousins $(\eta_+,\eta_-)$ in blue and red, respectively.} \label{constant}
\end{figure}

\section[The T-transform]{The $\boldsymbol{\mathcal{T}}$-transform}\label{section3}

In this section, we will introduce a geometric transformation on null curves in $\AdS$, which we call the \emph{$\mathcal{T}$-transform}, and show how to build it using a solution of the Riccati equation.

In order to define the $\mathcal{T}$-transform on null curves, we will employ Tabachnikov's transformation\footnote{In~\cite{T}, a symmetric relation is introduced for circle maps by considering the continuous version of a relation for ideal polygons in the hyperbolic plane. Extending this relation to the lifts of these circle maps to $\mathbb{R}^2$, a~transformation is defined. This transformation is, essentially, our $\mathcal{T}$-transform on star-shaped curves.} on star-shaped curves~\cite{T} and translate it to null curves through the associated pair of cousins.

\begin{Definition}\label{5.1} Two star-shaped curves $\eta,\widetilde{\eta}\colon J\subseteq\R\longrightarrow\dot{\R}^2$ parameterized by the central affine arc length are said to be \emph{$\mathcal{T}$-transforms of each other}\footnote{Clearly, from Definition~\ref{5.1} it follows that the $\mathcal{T}$-transform is reciprocal. That is, if $\eta$ is a $\mathcal{T}$-transform of $\widetilde{\eta}$, then $\widetilde{\eta}$ is also a $\mathcal{T}$-transform of $\eta$. Equivalently, $\eta$ and $\widetilde{\eta}$ are $\mathcal{T}$-transforms of each other. We will use the different terms indistinctively.} if $\det (\eta,\widetilde{\eta})$ is a nonzero constant.
\end{Definition}

\begin{Remark} The $\mathcal{T}$-transform of star-shaped curves is a geometric transformation, in the sense that two star-shaped curves are $\mathcal{T}$-transforms of each other if the area of the triangles with vertices $O=(0,0)$, $\eta(s)$, and $\widetilde{\eta}(s)$ is constant for every $s\in J$. In Figures~\ref{transform+} and~\ref{transform-}, we show a couple of examples of star-shaped curves and their corresponding $\mathcal{T}$-transforms and illustrate this geometric property.
\end{Remark}

\begin{figure}[t]\centering
			\includegraphics[height=4cm,width=4cm]{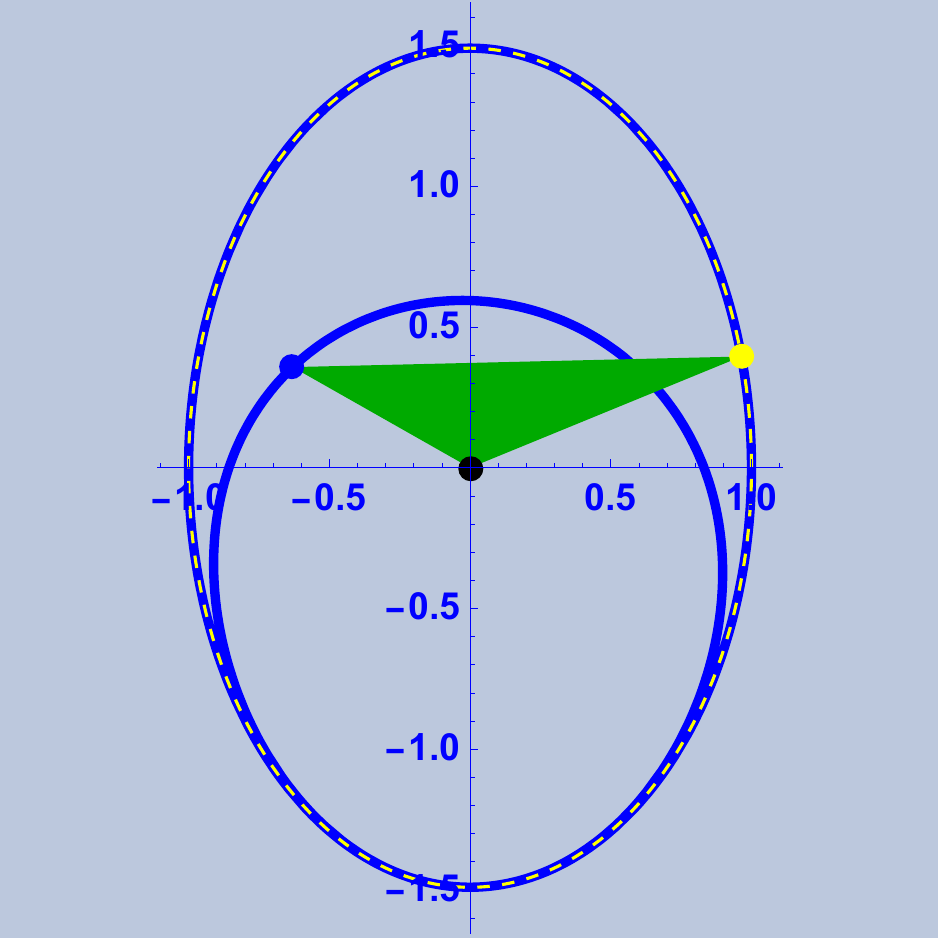}\quad\quad\quad
			\includegraphics[height=4cm,width=4cm]{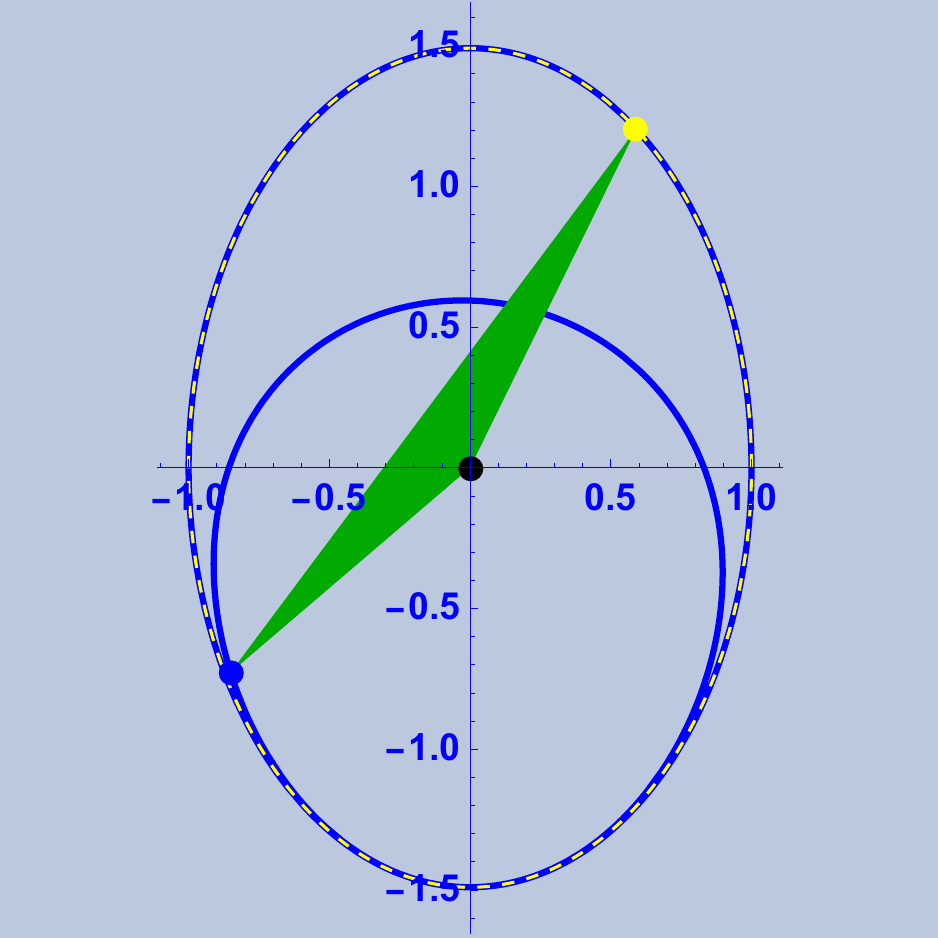}
\caption{The star-shaped curve $\eta_+$ of Figure~\ref{constant}, in blue and dotted in yellow, and the $\mathcal{T}$-transform~$\widetilde{\eta}_+$ of~$\eta_+$ in plain blue. The green triangles, which have the same area, are the ones with vertices $O$ (in~black), $\eta_+(s)$ (in~yellow) and \smash{$\widetilde{\eta}_+(s)$} (in blue). For each figure the triangle is shown at different values of the parameter $s\in J$.} \label{transform+}
\end{figure}

\begin{figure}[t]\centering
			\includegraphics[height=4cm,width=4cm]{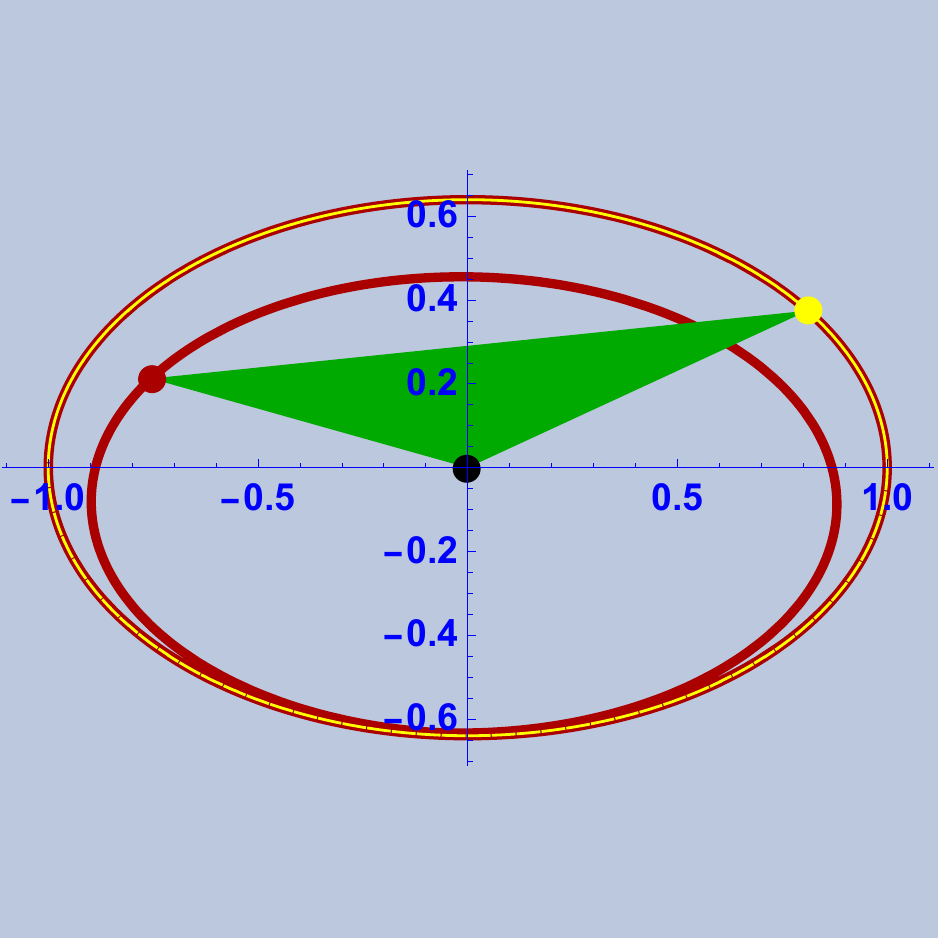}\quad\quad\quad
			\includegraphics[height=4cm,width=4cm]{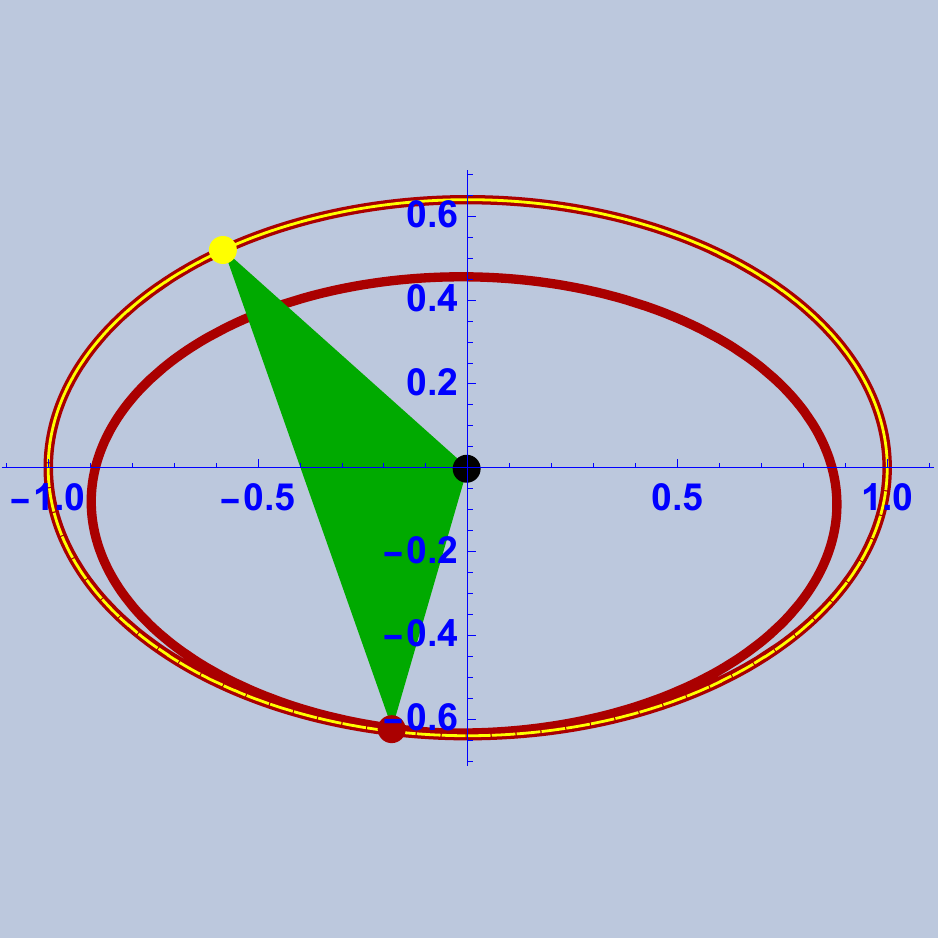}
\caption{The star-shaped curve $\eta_-$ of Figure~\ref{constant}, in red and dotted in yellow, and the $\mathcal{T}$-transform~$\widetilde{\eta}_-$ of~$\eta_-$ in plain red. The green triangles, which have the same area, are the ones with vertices $O$ (in~black), $\eta_-(s)$ (in yellow) and $\widetilde{\eta}_-(s)$ (in red). For each figure the triangle is shown at different values of the parameter $s\in J$.} \label{transform-}
\end{figure}

\begin{Definition}\label{5.3}
	Let $\gamma,\widetilde{\gamma}\colon J\subseteq\R\longrightarrow\AdS$ be two null curves with associated pairs of cousins $(\eta_+,\eta_-)$ and $(\widetilde{\eta}_+,\widetilde{\eta}_-)$, respectively. We say that $\gamma$ and $\widetilde{\gamma}$ are \emph{$\mathcal{T}$-transforms of each other} if $\eta_+$ is a $\mathcal{T}$-transform of $\widetilde{\eta}_+$ and $\eta_-$ is a $\mathcal{T}$-transform of $\widetilde{\eta}_-$.
\end{Definition}

\begin{figure}[t]\centering
			\includegraphics[height=4cm,width=4cm]{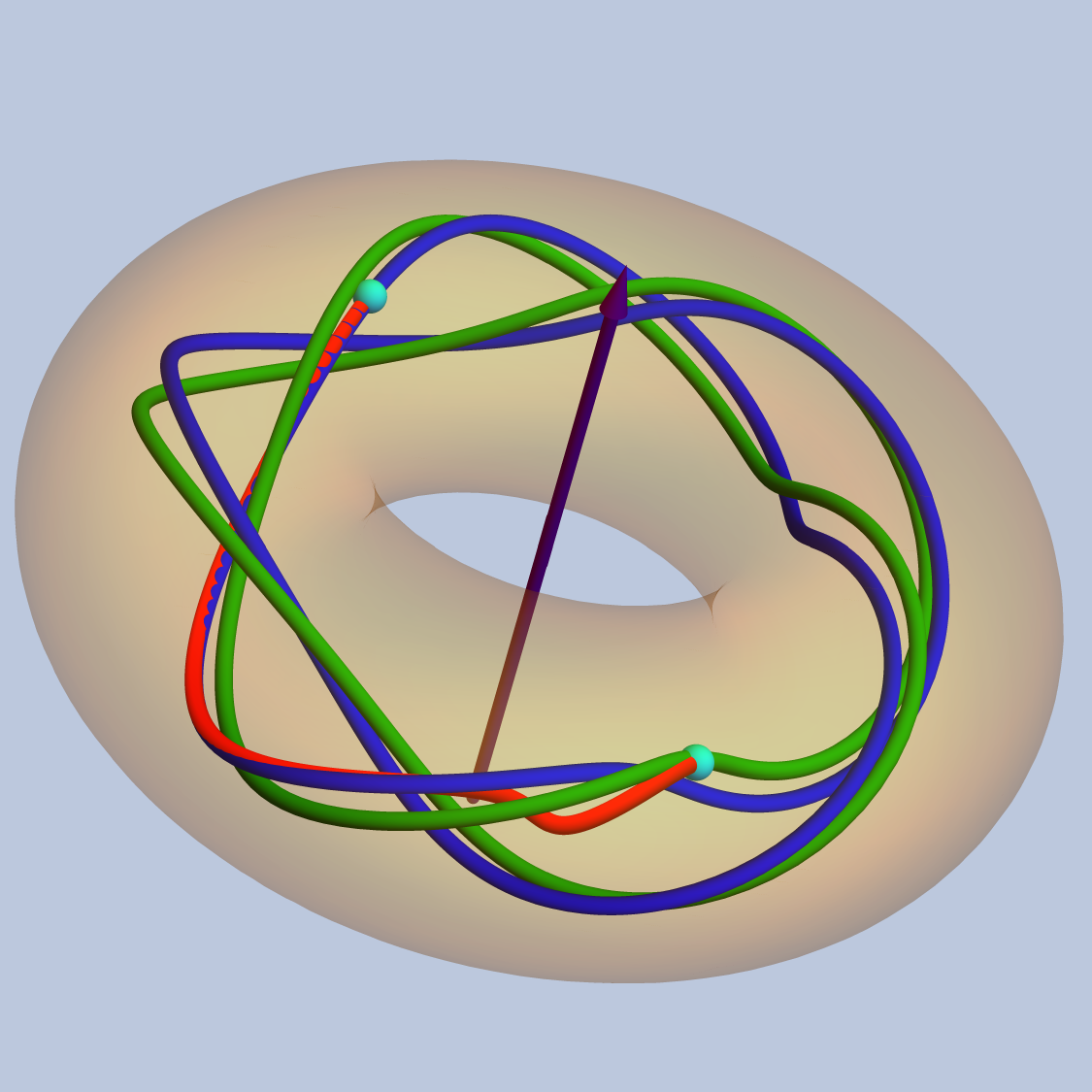}\quad\quad\quad
			\includegraphics[height=4cm,width=4cm]{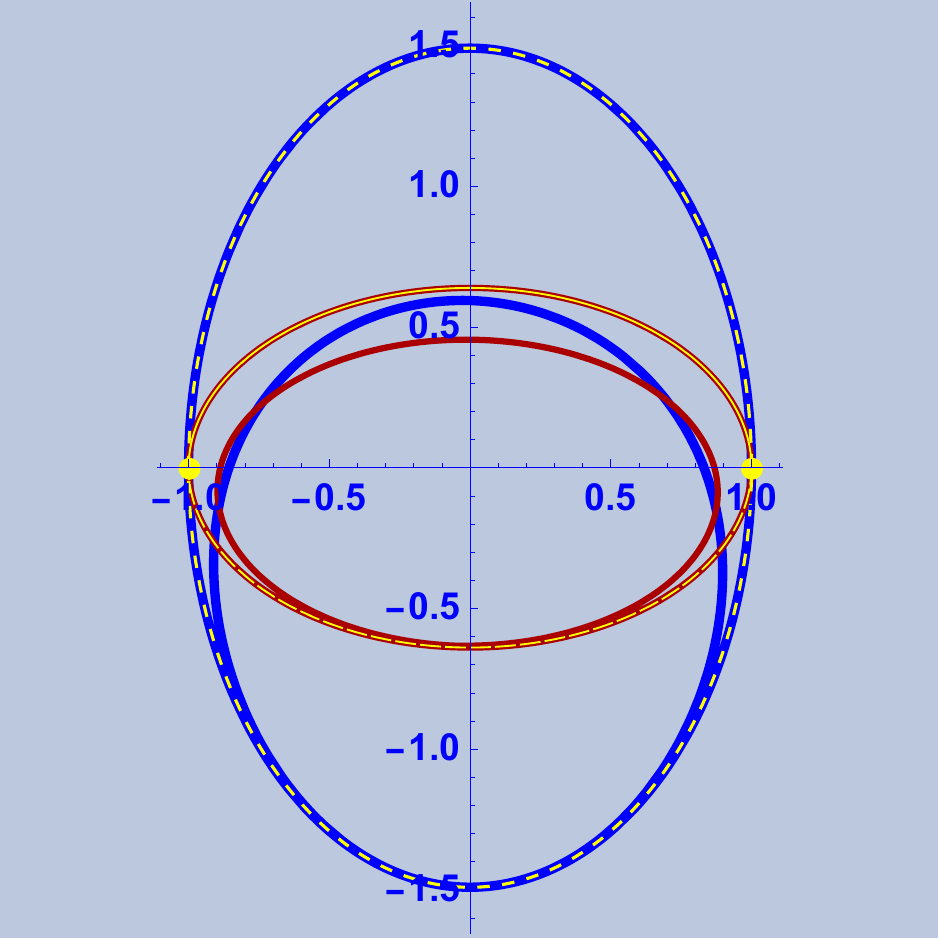}
\caption{Left: The $\mathcal{T}$-transform (for $\chi=0$) $\widetilde{\gamma}=\mathcal{T}_{\xi,f}(\gamma)$ of the null curve $\gamma$ with constant bending $\bending_{7,3}$ (see Figure~\ref{constant}). The parameter $\xi\neq 0$ of the $\mathcal{T}$-transform is $\xi=1.01$ and the transforming function $f$ is the solution of \eqref{Riccati} with initial condition $0.1$. Right: The associated pair of star-shaped cousins $(\widetilde{\eta}_+,\widetilde{\eta}_-)$ in plain blue and red, respectively. The curves dotted in yellow are the corresponding $\mathcal{T}$-transforms, namely, $\eta_+$ and $\eta_-$ (cf.\ Figures~\ref{transform+} and~\ref{transform-}).} \label{transform}
\end{figure}

\begin{Remark}
	In Figures~\ref{transform+} and~\ref{transform-}, we have shown the star-shaped $\mathcal{T}$-transforms $\widetilde{\eta}_+$ and $\widetilde{\eta}_-$ of the pair of cousins $(\eta_+,\eta_-)$ associated with the null curve $\gamma$ with constant bending $\bending_{7,3}$ (see Figure~\ref{constant}). Hence, according to Definition~\ref{5.3}, the null curve $\widetilde{\gamma}$ with associated pair of cousins~$(\widetilde{\eta}_+,\widetilde{\eta}_-)$ is a~$\mathcal{T}$-transform of $\gamma$. The curve $\widetilde{\gamma}$ is shown in Figure~\ref{transform}.
\end{Remark}

Using the definition of the $\mathcal{T}$-transform on star-shaped curves, we next find an equation satisfied by null curves that are $\mathcal{T}$-transform of each other.

\begin{Proposition}\label{T1} Let $\gamma,\widetilde{\gamma}\colon J\subseteq\R\longrightarrow\AdS$ be two null curves with associated pairs of cousins $(\eta_+,\eta_-)$ and $(\widetilde{\eta}_+,\widetilde{\eta}_-)$, respectively. If $\gamma$ and $\widetilde{\gamma}$ are $\mathcal{T}$-transforms of each other, then
	\begin{equation}\label{chi}
		\det \bigl(\eta_+,\widetilde{\eta}_+\bigr)\det \bigl(\eta_-,\widetilde{\eta}_-'\bigr)-\det \bigl(\eta_-,\widetilde{\eta}_-\bigr)\det \bigl(\eta_+,\widetilde{\eta}_+'\bigr)=\chi
	\end{equation}
	holds for some constant $\chi\in\R$.
\end{Proposition}
\begin{proof} Let $\gamma$ and $\widetilde{\gamma}$ be two null curves that are $\mathcal{T}$-transforms of each other. It then follows from Definition~\ref{5.3} that $\eta_+$ is a $\mathcal{T}$-transform of $\widetilde{\eta}_+$ and $\eta_-$ is a $\mathcal{T}$-transform of $\widetilde{\eta}_-$. Moreover, Defi\-ni\-tion~\ref{5.1} then guarantees the existence of two nonzero constants $c_\pm\in\R$ such that ${\det (\eta_+,\widetilde{\eta}_+)=c_+}$ and $\det (\eta_-,\widetilde{\eta}_-)=c_-$, respectively.
	
	Then there exist smooth functions $f_{\pm}\colon J\subseteq\R\longrightarrow \R$ such that, respectively,
	\begin{gather}
		\widetilde{\eta}_+=\frac{1}{c_+}\bigl(\eta_+'-f_+\eta_+\bigr),\label{T1:f1}\\
		\widetilde{\eta}_-=\frac{1}{c_-}\bigl(\eta_-'-f_-\eta_-\bigr).\label{T1:f12}
	\end{gather}
	From now on, we only work with \eqref{T1:f1} and the sub-index $+$. The argument for \eqref{T1:f12} and the sub-index $-$ is analogous.
	
	Differentiating \eqref{T1:f1} and using the spinorial Frenet-type equations \eqref{dF+} (respectively, \eqref{dF-} when the sub-index is $-$) of $\gamma$, we obtain
	\begin{equation}\label{T1:f2}
		\widetilde{\eta}_+'=\frac{1}{c_+}\bigl(\bigl[\bending+1-f_+'\bigr]\eta_+-f_+\eta_+'\bigr),
	\end{equation}
	where $\bending$ is the bending of $\gamma$. Since $\eta_+$ and $\widetilde{\eta}_+$ are two star-shaped curves parameterized by the central affine arc length, $\det \bigl(\eta_+,\eta_+'\bigr)=1$ and $\det \bigl(\widetilde{\eta}_+,\widetilde{\eta}_+'\bigr)=1$. Thus, it follows from~\eqref{T1:f1} and~\eqref{T1:f2} that
	\begin{align*}
		1=\det \bigl(\widetilde{\eta}_+,\widetilde{\eta}_+'\bigr) =\frac{1}{c_+^2}\det \bigl(\eta_+'-f_+\eta_+,\bigl[\bending+1-f_+'\bigr]\eta_+-f_+\eta_+'\bigr)
		 =\frac{1}{c_+^2}\bigl(f_+'-\bending-1+f_+^2\bigr).
	\end{align*}
	Therefore, the function $f_+$ is a solution of the Riccati equation
	\begin{equation}\label{T1:f3}
		f_+'+f_+^2=\bending+1+c_+^2.
	\end{equation}
	Combining \eqref{T1:f1}, \eqref{T1:f2} and \eqref{T1:f3}, we obtain that the canonical central affine frame field $\widetilde{F}_+$ along~$\widetilde{\eta}_+$ is given in terms of the canonical central affine frame field $F_+$ along $\eta_+$ by
	\begin{equation}\label{T1:f4}
		\widetilde{F}_+=\frac{1}{c_+}F_+\begin{pmatrix} -f_+ & f_+^2-c_+^2 \\ 1 & -f_+\end{pmatrix}.
	\end{equation}
	Since it will be used later, we also specify here the expression for the canonical central affine frame field $\widetilde{F}_-$ along $\widetilde{\eta}_-$ in terms of $F_-$, namely,
	\begin{equation}\label{T1:f4-}
		\widetilde{F}_-=\frac{1}{c_-}F_-\begin{pmatrix} -f_- & f_-^2-c_-^2 \\ 1 & -f_-\end{pmatrix}.
	\end{equation}
	From the above relation \eqref{T1:f4} between $\widetilde{F}_+$ and $F_+$ (respectively, \eqref{T1:f4-}), the spinorial Frenet-type equations \eqref{dF+} (respectively, \eqref{dF-}) and the Riccati equation \eqref{T1:f3} satisfied by the function $f_+$ (respectively, the analogue one for $f_-$), we conclude with
	\begin{align*}
		\widetilde{F}_+^{-1}{\rm d}\widetilde{F}_+&{}=\frac{1}{c_+^2}\begin{pmatrix} -f_+ & c_+^2-f_+^2 \\ -1 & -f_+ \end{pmatrix}\biggl[F_+^{-1}{\rm d}F_+\begin{pmatrix} -f_+ & f_+^2-c_+^2 \\ 1 & -f_+\end{pmatrix}+\begin{pmatrix} -f_+' & (f_+^2)' \\ 0 & -f_+'\end{pmatrix} \biggr]\\
		&{}=\begin{pmatrix} 0 & 2f_+^2-\bending-1-2c_+^2 \\ 1 & 0 \end{pmatrix}.
	\end{align*}
	Hence, from \eqref{dF+} (respectively, \eqref{dF-}), the bending $\widetilde{\bending}$ of $\widetilde{\gamma}$ satisfies
	\begin{equation}\label{T1:f6}
		\widetilde{\bending}=2f_+^2-\bending-2-2c_+^2.
	\end{equation}
	The analogue computations for the sub-index $-$ show that $\widetilde{\bending}$ can also be expressed as
	\begin{equation}\label{T1:f62}
		\widetilde{\bending}=2f_-^2-\bending+2-2c_-^2.
	\end{equation}
	Therefore, setting equal \eqref{T1:f6} and \eqref{T1:f62}, it follows that
	\begin{equation*}%\label{T1:f7}
		f_+^2-f_-^2=2+c_+^2-c_-^2.
	\end{equation*}
	We now distinguish between three cases: $2+c_+^2-c_-^2>0$, $2+c_+^2-c_-^2<0$, and $2+c_+^2-c_-^2=0$.
	
	Assume first that $2+c_+^2-c_-^2>0$ holds. Then, there exists a smooth function $\phi$ such that
	\begin{equation*}
		f_+=\pm\sqrt{2+c_+^2-c_-^2}\cosh(\phi),\qquad
		f_-=\sqrt{2+c_+^2-c_-^2}\sinh(\phi).
	\end{equation*}
	This function $\phi$ satisfies two differential equations arising from the Riccati equation \eqref{T1:f3} satisfied by $f_+$ and the analogue one for $f_-$, namely,
	\begin{gather*}
\begin{split}
&\pm\sqrt{2+c_+^2-c_-^2}\sinh(\phi)\phi'+\bigl(2+c_+^2-c_-^2\bigr)\cosh^2(\phi)=\bending+1+c_+^2,\\
&\sqrt{2+c_+^2-c_-^2}\cosh(\phi)\phi'+\bigl(2+c_+^2-c_-^2\bigr)\sinh^2(\phi)=\bending-1+c_-^2.
\end{split}
	\end{gather*}
	Subtracting both equations, we then see that $\phi'=0$ and, hence, $\phi$ is constant. This implies that both $f_+$ and $f_-$ are also constant functions and, in addition, $f_+\neq f_-$.
	
	Finally, we use \eqref{T1:f4} (and \eqref{T1:f4-}) together with the fact that $f_+$ and $f_-$ are constant functions to compute \eqref{chi}, obtaining
	\begin{align*}
		\chi&{}=\det (\eta_+,\widetilde{\eta}_+)\det \bigl(\eta_-,\widetilde{\eta}_-'\bigr)-\det (\eta_-,\widetilde{\eta}_-)\det \bigl(\eta_+,\widetilde{\eta}_+'\bigr)\\
		&{}=\frac{1}{c_+}\biggl(\frac{-f_-}{c_-}\biggr)-\frac{1}{c_-}\biggl(\frac{-f_+}{c_+}\biggr)=\frac{f_+-f_-}{c_+c_-},
	\end{align*}
	which is a nonzero constant.
	
	The case $2+c_+^2-c_-^2<0$ follows the same reasoning, hence, we avoid repeating it here.
	
	Assume finally that $2+c_+^2-c_-^2=0$. Then, $f_+^2=f_-^2$ holds. From the Riccati equations (\eqref{T1:f3} and the analogue one) satisfied by $f_+$ and $f_-$, we compute
	\begin{equation*}f_+'-f_-'=2+c_+^2-c_-^2-f_+^2+f_-^2=0.\end{equation*}
	Thus, integrating this we obtain, $f_+=f_-+\ell$, for some constant $\ell\in\R$. Observe that since $f_+^2=f_-^2$, we have two options: $f_+=f_-$ when $\ell=0$, or $f_+=-f_-=\ell/2$ when $\ell\neq 0$. In both cases, we compute \eqref{chi} as in the previous cases obtaining
	\begin{equation*}\chi=\det (\eta_+,\widetilde{\eta}_+)\det \bigl(\eta_-,\widetilde{\eta}_-'\bigr)-\det (\eta_-,\widetilde{\eta}_-)\det \bigl(\eta_+,\widetilde{\eta}_+'\bigr)=\frac{f_+-f_-}{c_+c_-}=\frac{\ell}{c_+c_-},\end{equation*}
	which is a constant. This concludes the proof.
\end{proof}

Equation \eqref{chi} gives two essentially different possibilities for the $\mathcal{T}$-transform on null curves, which depend on whether $\chi=0$ or $\chi\neq 0$.

\subsection[Case chi neq 0]{Case $\boldsymbol{\chi\neq 0}$}

We will show that the most interesting case is $\chi=0$, since for the case $\chi\neq 0$ the bending of the null curves is constant.

\begin{Proposition}\label{5.5} Let $\gamma\colon J \subseteq \R\longrightarrow \AdS$ be a null curve with bending $\bending$ and $(F_+,F_-)$ be the spinor frame field along $\gamma$. If $\gamma$ has $\mathcal{T}$-transforms $\widetilde{\gamma}$ satisfying \eqref{chi} with $\chi\neq 0$, then the bending~$\bending$ of $\gamma$ is constant and the $\mathcal{T}$-transforms $\widetilde{\gamma}$ are given by
\begin{equation*}\widetilde{\gamma}(s)=\frac{1}{c_+c_-}F_+\begin{pmatrix}-f_+&\bending+1\\1&-f_+\end{pmatrix}\begin{pmatrix}-f_-&1-\bending\\-1&-f_-\end{pmatrix}F_-^{-1},\end{equation*}
	where $c_+,c_-\in\R$ and $f_+\neq f_-$ are two constants such that
$f_+^2=\bending+1+c_+^2$, $f_-^2=\bending-1+c_-^2$.
Moreover, the $\mathcal{T}$-transforms $\widetilde{\gamma}$ of $\gamma$ have constant bending $\widetilde{\bending}=\bending$.
\end{Proposition}
\begin{proof} From the proof of Proposition~\ref{T1}, it follows that if $\chi\neq 0$, then both $f_{\pm}$ are constant functions. Hence, from the Riccati equation \eqref{T1:f3} (and the analogue one for $f_-$), we have
$f_+^2=\bending+ 1+c_+^2$, $f_-^2=\bending-1+c_-^2$.
	This proves that the bending $\bending$ of $\gamma$ must be constant. Moreover, using this in the relation between $\bending$ and the bending $\widetilde{\bending}$ of $\widetilde{\gamma}$ given in \eqref{T1:f6}, we conclude that $\widetilde{\bending}=\bending$.
	
	Furthermore, since according to Theorem~\ref{relation} $\widetilde{\gamma}=\widetilde{F}_+\widetilde{F}_-^{-1}$,
	from the expression of the central affine frame field $\widetilde{F}_+$ given in \eqref{T1:f4} and the analogue one for~$\widetilde{F}_-$ given in \eqref{T1:f4-} it is a straightforward computation to obtain the explicit expression of the $\mathcal{T}$-transforms satisfying \eqref{chi} with $\chi\neq 0$.
\end{proof}

\begin{Remark} From Proposition~\ref{5.5} if $\gamma$ and $\widetilde{\gamma}$ are $\mathcal{T}$-transforms of each other for $\chi\neq 0$, then they have the same constant bending. However, the converse is not true. Null curves with constant bending have $\mathcal{T}$-transforms for $\chi=0$ and, in particular, these $\mathcal{T}$-transforms may not have constant bending (see, for instance, Figure~\ref{transform} and/or Example~\ref{5.14}).
\end{Remark}

\subsection[Case chi = 0]{Case $\boldsymbol{\chi=0}$}

From now on, we focus on $\mathcal{T}$-transforms satisfying \eqref{chi} for $\chi=0$. The next result shows how to construct $\mathcal{T}$-transforms for $\chi=0$ of a null curve $\gamma$ beginning with solutions of a Riccati equation.

\begin{Theorem}\label{T2} Let $\gamma\colon J\subseteq\R\longrightarrow\AdS$ be a null curve with bending $\bending$ and $(F_+,F_-)$ be the spinor frame field along it. A null curve $\widetilde{\gamma}\colon J\subseteq\R\longrightarrow \AdS$ is a $\mathcal{T}$-transform of $\gamma$ for $\chi=0$ if and only if
	\begin{equation}\label{gammatilda}
		\widetilde{\gamma}=\pm F_+\begin{pmatrix} \tanh(\xi) & -\csch(\xi)\sech(\xi)f \\ 0 & \coth(\xi) \end{pmatrix} F_-^{-1},
	\end{equation}
	where $\xi\neq 0$ is a constant and $f\colon J\subseteq\R\longrightarrow\R$ is a solution of the Riccati equation
	\begin{equation}\label{Riccati}
		f'+f^2=\bending+\cosh(2\xi).
	\end{equation}
	Moreover, the bending of $\widetilde{\gamma}$ is
	\begin{equation}\label{T2:f3}
		\widetilde{\bending}=-\bending+2f^2-2\cosh(2\xi).
	\end{equation}
\end{Theorem}
\begin{proof} Let $\gamma$ be a null curve with bending $\bending$ and assume that $\widetilde{\gamma}$ is a $\mathcal{T}$-transform of $\gamma$ for $\chi=0$.
	
	The proof of Proposition~\ref{T1} shows that $f_+=f_-$ and $2+c_+^2-c_-^2=0$. From the latter, there must exist a nonzero constant $\xi$ such that $c_+=\sqrt{2}\sinh(\xi)$ and $c_-=\pm\sqrt{2}\cosh(\xi)$. Moreover, the function $f_+=f_-$ is a solution of the Riccati equation \eqref{T1:f3}. For simplicity, we simply put $f=f_\pm$. The Riccati equation \eqref{T1:f3} now reads
	\begin{equation*}f'+f^2=\bending+1+2\sinh^2(\xi)=\bending+\cosh(2\xi),\end{equation*}
	which coincides with \eqref{Riccati} in the statement. Further, from the expression of the central affine frame field $\widetilde{F}_+$ given in \eqref{T1:f4}, the analogue one for $\widetilde{F}_-$ given in \eqref{T1:f4-}, $f=f_\pm$, and the above values of $c_+$ and $c_-$, we have
	\begin{align*}
		\widetilde{\gamma}=\widetilde{F}_+\widetilde{F}_-^{-1}&{}=\frac{1}{c_+c_-}F_+\begin{pmatrix} -f & f^2-c_+^2 \\ 1 & -f \end{pmatrix}\begin{pmatrix} -f & c_-^2-f^2 \\ -1 & -f\end{pmatrix}F_-^{-1}\\
		&{}=\pm F_+\begin{pmatrix} \tanh(\xi) & -\csch(\xi)\sech(\xi)f \\ 0 & \coth(\xi)\end{pmatrix}F_-^{-1},
	\end{align*}
	proving \eqref{gammatilda}. Finally, we deduce from \eqref{T1:f6} that
	\begin{equation*}\widetilde{\bending}=2f^2-\bending-2-4\sinh^2(\xi)=-\bending+2f^2-2\cosh(2\xi).\end{equation*}
	This concludes the forward implication.
	
	Conversely, suppose that $f\colon J\subseteq\R\longrightarrow\R$ is a solution of the Riccati equation \eqref{Riccati} and consider the $\SL\times\SL$-valued map defined by \smash{$\bigl(\widetilde{F}_+,\widetilde{F}_-\bigr)$}, where
	\begin{gather}
		\widetilde{F}_+=\frac{1}{\sqrt{2} \sinh(\xi)}\,F_+\begin{pmatrix} -f & f^2-2\sinh^2(\xi) \\ 1 & -f\end{pmatrix},\label{tildeF+}\\
		\widetilde{F}_-=\frac{\pm 1}{\sqrt{2} \cosh(\xi)}\,F_-\begin{pmatrix} -f & f^2-2\cosh^2(\xi) \\ 1 & -f \end{pmatrix}.\label{tildeF-}
	\end{gather}
	Recall that $(F_+,F_-)$ is the spinor frame field along the null curve $\gamma$.
	
	Then, the map $\bigl(\widetilde{F}_+,\widetilde{F}_-\bigr)$ is a lift of a null curve $\widetilde{\gamma}$ to $\SL\times\SL$. We will first show that \smash{$\bigl(\widetilde{F}_+,\widetilde{F}_-\bigr)$} is indeed the spinor frame field along $\widetilde{\gamma}$.
	
	From the spinorial Frenet-type equations \eqref{dF+} and \eqref{dF-} of $\gamma$ and the definition of $\widetilde{F}_\pm$ given in \eqref{tildeF+} and \eqref{tildeF-}, it follows that
	\begin{gather*}
		\widetilde{F}_+^{-1}{\rm d}\widetilde{F}_+=\begin{pmatrix} 0 & \widetilde{\bending}+1 \\ 1 & 0 \end{pmatrix},\qquad
		\widetilde{F}_-^{-1}{\rm d}\widetilde{F}_-=\begin{pmatrix} 0 & \widetilde{\bending}-1 \\ 1 & 0 \end{pmatrix},
	\end{gather*}
	where $\widetilde{\bending}$ is as in \eqref{T2:f3}. Consequently, \smash{$\bigl(\widetilde{F}_+,\widetilde{F}_-\bigr)$} is the spinor frame field along $\widetilde{\gamma}$ and the function~$\widetilde{\bending}$ is the bending of $\widetilde{\gamma}$.
	
	We now prove that $\widetilde{\gamma}$ is a $\mathcal{T}$-transform of $\gamma$. Since $\bigl(\widetilde{F}_+,\widetilde{F}_-\bigr)$ is the spinor frame field along~$\widetilde{\gamma}$, the first column vectors of \smash{$\widetilde{F}_+$} and \smash{$\widetilde{F}_-$} are the pair $(\widetilde{\eta}_+,\widetilde{\eta}_-)$ of star-shaped cousins associated with $\widetilde{\gamma}$. From the expressions of $\widetilde{F}_+$ and $\widetilde{F}_-$ given in \eqref{tildeF+} and \eqref{tildeF-}, respectively, we have
	\begin{gather*}
		\widetilde{\eta}_+=\frac{1}{\sqrt{2} \sinh(\xi)}\bigl(-f\eta_++\eta_+'\bigr),\qquad
		\widetilde{\eta}_-=\frac{\pm 1}{\sqrt{2} \cosh(\xi)}\bigl(-f\eta_-+\eta_-'\bigr).
	\end{gather*}
	A simple computation involving $\det \bigl(\eta_+,\eta_+'\bigr)=1$ and $\det \bigl(\eta_-,\eta_-'\bigr)=1$, then shows that
	\begin{equation*}
		\det (\eta_+,\widetilde{\eta}_+)=\frac{1}{\sqrt{2} \sinh(\xi)},\qquad \det (\eta_-,\widetilde{\eta}_-)=\frac{\pm 1}{\sqrt{2} \cosh(\xi)}.
	\end{equation*}
	From Definition~\ref{5.1}, we conclude that $\eta_+$ is a $\mathcal{T}$-transform of $\widetilde{\eta}_+$ and that $\eta_-$ is a $\mathcal{T}$-transform of $\widetilde{\eta}_-$. Finally, according to Definition~\ref{5.3}, $\widetilde{\gamma}$ is a $\mathcal{T}$-transform of $\gamma$.
\end{proof}

\begin{Definition} The null curve $\widetilde{\gamma}$ given by \eqref{gammatilda} is called the \emph{$\mathcal{T}$-transform $($for $\chi=0)$ of $\gamma$ with parameter $\xi$ and transforming function $f$}. We denote it by $\mathcal{T}_{\xi,f}(\gamma)$.
\end{Definition}

Figure~\ref{transform} shows an example of a $\mathcal{T}$-transform (for $\chi=0$) $\widetilde{\gamma}=\mathcal{T}_{\xi,f}(\gamma)$ of a null curve $\gamma$ computed as in \eqref{gammatilda} of Theorem~\ref{T2}. The initial null curve $\gamma$ is the closed null curve with constant bending $\bending_{7,3}$ illustrated in Figure~\ref{constant}. The parameter $\xi\neq 0$ is a fixed real number and the transforming function is obtained by solving the Riccati equation \eqref{Riccati} for a fixed initial condition.

In the next result, we prove, using standard arguments, a permutability theorem for the $\mathcal{T}$-transform on null curves in $\AdS$.

\begin{Theorem}\label{T3} Let $\gamma\colon J\subseteq\R\longrightarrow\AdS$ be a null curve with bending $\bending$ and consider two $\mathcal{T}$-transforms of $\gamma$, namely, $\gamma_1=\mathcal{T}_{\xi_1,f_1}(\gamma)$ and $\gamma_2=\mathcal{T}_{\xi_2,f_2}(\gamma)$ with parameters $\xi_1\neq\xi_2$ and transforming functions $f_1\neq f_2$, respectively. Then, the functions
	\begin{gather*}
		f_{2\to 1}=-f_1+\frac{\cosh(2\xi_1)-\cosh(2\xi_2)}{f_1-f_2},\qquad
		f_{1\to 2}=-f_2+\frac{\cosh(2\xi_1)-\cosh(2\xi_2)}{f_1-f_2},
	\end{gather*}
	satisfy, respectively, the Riccati equations
	\begin{equation*}
		f_{2\to 1}'+f_{2\to 1}^2=\bending_1+\cosh(2\xi_2),\qquad
		f_{1\to 2}'+f_{1\to 2}^2=\bending_2+\cosh(2\xi_1),
	\end{equation*}
	where $\bending_1$ and $\bending_2$ are the bendings of $\gamma_1$ and $\gamma_2$, respectively. Moreover,
	\begin{equation*}%\label{T3:f3}
		\mathcal{T}_{\xi_2,f_{2\to 1}}(\mathcal{T}_{\xi_1,f_1}(\gamma))=\mathcal{T}_{\xi_1,f_{1\to 2}}(\mathcal{T}_{\xi_2,f_2}(\gamma)),
	\end{equation*}
	and the bending of $\mathcal{T}_{\xi_2,f_{2\to 1}}(\mathcal{T}_{\xi_1,f_1}(\gamma))$ is the function defined by
	\begin{equation*}%\label{T3:f4}
		\bending_{2\to 1}=\bending-2(\cosh(2\xi_1)-\cosh(2\xi_2))\frac{f_1+f_2}{f_1-f_2}+2\biggl(\frac{\cosh(2\xi_1)-\cosh(2\xi_2)}{f_1-f_2}\biggr)^2.
	\end{equation*}
	$($Of course, the bending $\bending_{1\to 2}$ of $\mathcal{T}_{\xi_1,f_{1\to 2}}(\mathcal{T}_{\xi_2,f_2}(\gamma))$ is, precisely, $\bending_{1\to 2}=\bending_{2\to 1}$.$)$
\end{Theorem}
\begin{proof} Let $\gamma$ be a null curve with bending $\bending$ and assume that $\gamma_1=\mathcal{T}_{\xi_1,f_1}(\gamma)$ and $\gamma_2=\mathcal{T}_{\xi_2,f_2}(\gamma)$ are two $\mathcal{T}$-transforms of $\gamma$. From Theorem~\ref{T2}, it follows that
	\begin{equation*}
		\bending_1=-\bending+2f_1^2-2\cosh(2\xi_1),\qquad
		\bending_2=-\bending+2f_2^2-2\cosh(2\xi_2)
	\end{equation*}
	are the bendings of $\gamma_1$ and $\gamma_2$, respectively. In addition, the transforming functions $f_1$ and $f_2$ are solutions of the Riccati equations
	\begin{equation*}
		f_1'+f_1^2=\bending+\cosh(2\xi_1),\qquad
		f_2'+f_2^2=\bending+\cosh(2\xi_2),
	\end{equation*}	
	respectively.
	
	Then, an elementary computation involving above differential equations shows that the functions $f_{2\to1}$ and $f_{1\to 2}$ of the statement satisfy the desired Riccati equations. Therefore, $f_{2\to 1}$ is a~transforming function of $\gamma_1$ with parameter $\xi_2$, while $f_{1\to 2}$ is a transforming function of $\gamma_2$ with parameter $\xi_1$.
	
	For simplicity, define the $\SL$-valued maps
	\begin{gather*}
		G_+(\xi,\phi)=\frac{1}{\sqrt{2}\sinh(\xi)}\begin{pmatrix} -\phi & \phi^2-2\sinh^2(\xi) \\ 1 & -\phi \end{pmatrix},\\
		G_-(\xi,\phi)=\frac{\pm1}{\sqrt{2}\cosh(\xi)}\begin{pmatrix} -\phi & \phi^2-2\cosh^2(\xi) \\ 1 & -\phi \end{pmatrix}.
	\end{gather*}
	As shown in the proof of Theorem~\ref{T2}, the maps $(F_+G_+(\xi,\phi),F_-G_-(\xi,\phi))$, for suitable values $\xi\neq 0$ and solutions $\phi$ of the corresponding Riccati equations, are the spinor frame fields along $\mathcal{T}$-transforms of null curves with spinor frame field $(F_+,F_-)$. In particular, we have that $(F_+G_+(\xi_1,f_1),F_-G_-(\xi_1,f_1))$ is the spinor frame field along $\gamma_1$ and $(F_+G_+(\xi_2,f_2),F_-G_-(\xi_2,f_2))$ is the spinor frame field along $\gamma_2$. Observe that, recursively, it follows that the map
	\begin{equation*}([F_+G_+(\xi_1,f_1)]G_+(\xi_2,f_{2\to 1}),[F_-G_-(\xi_1,f_1)]G_-(\xi_2,f_{2\to 1}))\end{equation*}
	is the spinor frame field along $\mathcal{T}_{\xi_2,f_{2\to 1}}(\gamma_1)$ since $(F_+G_+(\xi_1,f_1),F_-G_-(\xi_1,f_1))$ is the spinor frame field along $\gamma_1$. Similarly,
	\begin{equation*}([F_+G_+(\xi_2,f_2)]G_+(\xi_1,f_{1\to 2}),[F_-G_-(\xi_2,f_2)]G_-(\xi_1,f_{1\to 2}))\end{equation*}
	is the spinor frame field along $\mathcal{T}_{\xi_1,f_{1\to 2}}(\gamma_2)$.
	
	It is now a computational matter to check that
	\begin{gather*}
		G_+(\xi_1,f_1)G_+(\xi_2,f_{2\to 1})=G_+(\xi_2,f_2)G_+(\xi_1,f_{1\to 2}),\\
		G_-(\xi_1,f_1)G_-(\xi_2,f_{2\to 1})=G_-(\xi_2,f_2)G_-(\xi_1,f_{1\to 2}),
	\end{gather*}
	which implies that $\mathcal{T}_{\xi_2,f_{2\to 1}}(\gamma_1)=\mathcal{T}_{\xi_1,f_{1\to 2}}(\gamma_2)$ holds.
	
	It remains to prove the expression of the bending $\bending_{2\to 1}$ of $\mathcal{T}_{\xi_2,f_{2\to 1}}(\gamma_1)$. Since $\mathcal{T}_{\xi_2,f_{2\to 1}}(\gamma_1)$ is a~$\mathcal{T}$-transform of $\gamma_1$ with parameter $\xi_2$ and transforming function $f_{2\to 1}$, it follows from Theorem~\ref{T2} that
	\begin{equation}\label{bendingT}
		\bending_{2\to 1}=-\bending_1+2f_{2\to 1}^2-2\cosh(2\xi_2).
	\end{equation}
	At the same time, $\gamma_1=\mathcal{T}_{\xi_1,f_1}(\gamma)$ is a $\mathcal{T}$-transform of $\gamma$ with parameter $\xi_1$ and transforming function $f_1$. Hence, $\bending_1=-\bending+2f_1^2-2\cosh(2\xi_1)$ which we substitute in \eqref{bendingT}, obtaining
	\begin{equation*}\bending_{2\to 1}=\bending-2f_1^2+2\cosh(2\xi_1)+2f_{2\to 1}^2-2\cosh(2\xi_2).\end{equation*}
	The result then follows immediately from the definition of $f_{2\to 1}$.
\end{proof}

\section{Geometric realization of the B\"{a}cklund transformation \\ for the KdV}\label{section4}

In this section, we will use the $\mathcal{T}$-transform on null curves to show a geometric realization of the B\"{a}cklund transformation for the KdV equation. We first briefly recall some basic facts about this B\"{a}cklund transformation.

\subsection{B\"{a}cklund transformation for the KdV}\label{section4.1}

The Korteweg--de Vries (KdV) equation is the PDE given by
\begin{equation}\label{KdV2}
	\partial_t\bending+\partial_s^3\bending-6\bending \partial_s\bending=0.
\end{equation}

Let $J,I\subseteq\R$ be two open intervals (for convenience, we assume $0\in J,I$) and consider a~solution $\bending\colon (s,t)\in J\times I\subseteq\R^2\longmapsto\bending(s,t)\in\R$ of the KdV equation \eqref{KdV2}. The \emph{B\"{a}cklund transform} of $\bending$ with spectral parameter $\lambda\in\R\setminus\{0\}$ and transforming function $f\colon J\times I\subseteq\R^2\longrightarrow\R$ is the function defined by
\begin{equation}\label{backlund}
	\widetilde{\bending}=-\bending+2f^2-2\lambda.
\end{equation}

In~\cite{WE}, Wahlquist and Estabrook showed that if the transforming function $f(s,t)$ is a solution of the so-called Wahlquist--Estabrook equation
\begin{equation}\label{EW}
	{\rm d}f=\bigl(\bending-f^2+\lambda\bigr){\rm d}s+\bigl(4\lambda\bigl[f^2-\lambda\bigr]-2\bigl[f^2+\lambda\bigr]+2\bending^2+2f\partial_s\bending-\partial_s^2\bending\bigr){\rm d}t,
\end{equation}
then the B\"{a}cklund transform \eqref{backlund} of $\bending$ with spectral parameter $\lambda\neq 0$ and transforming function~$f$ is another solution of the KdV equation \eqref{KdV2}.

The Wahlquist--Estabrook equation \eqref{EW} is an overdetermined system whose compatibility equation is the KdV equation \eqref{KdV2}. As a consequence, this system is locally solvable, in the sense that for every $(s_o,t_o)\in J\times I\subseteq\R^2$ and every constant $c\in\R$ there locally exists a unique solution of \eqref{EW} with initial condition $f(s_o,t_o)=c$.

We will next describe a procedure to construct this solution. This method will employ the extended frames of solutions of the KdV equation~\cite{TU}.

For a function $\bending\colon J\times I\subseteq\R^2\longrightarrow\R$ and a constant $\lambda\in\R$, define the $\mathfrak{sl}(2,\R)$-valued $1$-form
\begin{equation}\label{alpha}
	\Gamma_{\lambda}=\mathcal{K}_\lambda\,{\rm d}s+\mathcal{P}_\lambda\,{\rm d}t,
\end{equation}
where $\mathcal{K}_\lambda$ and $\mathcal{P}_\lambda$ are given by
\begin{equation}\label{KP}
	\mathcal{K}_\lambda=\begin{pmatrix} 0 & \bending+\lambda \\ 1 & 0\end{pmatrix}, \qquad \mathcal{P}_\lambda=\begin{pmatrix} -\partial_s\bending & -\partial_s^2\bending+2\bending^2-2\lambda\bending-4\lambda^2 \\ 2\bending-4\lambda & \partial_s\bending \end{pmatrix}.
\end{equation}
The $1$-form $\Gamma_{\lambda}$ satisfies the Maurer--Cartan compatibility equation if and only if the function~$\bending(s,t)$ satisfies the KdV equation \eqref{KdV2}. Consequently, as shown in~\cite{ZF}, for a given solution~$\bending$ of~\eqref{KdV2} and every $\lambda\in\R$ there exists a map $E_\lambda\colon J\times I\subseteq\R^2\longrightarrow\SL$ such that%
\begin{equation}\label{dE}
	{\rm d}E_\lambda=E_\lambda\,\Gamma_{\lambda},\qquad E_\lambda(0,0)={\rm Id}.
\end{equation}
The maps $E_\lambda$ depend in a real analytic fashion on $\lambda\in\R$. The map $E_\lambda$ is called an \emph{extended frame} of $\bending$ with spectral parameter $\lambda\in\R$~\cite{TU}.\footnote{Observe that in the paper~\cite{TU}, the spectral parameter $\lambda$ is a complex number, while here we are restricting it to real values.}

Consider the extended frames $E_\lambda$, $\lambda\in\R$, of a solution $\bending$ of the KdV equation \eqref{KdV2} and define\begin{equation}\label{TF}
	\begin{pmatrix} x \\ y \end{pmatrix}=E_\lambda^{-1}E_\lambda(s_o,t_o)\begin{pmatrix} -c \\ 1 \end{pmatrix}.
\end{equation}
Then, the function $f=-x/y$ is a solution\footnote{The function $f=-x/y$ is a local solution of \eqref{EW}. Indeed, this solution is singular at the zero locus of the function $y$.} of the Wahlquist--Estabrook equation \eqref{EW} with initial condition $f(s_o,t_o)=c$~\cite{TU}.

\begin{Remark}\label{WE-R}
	If $f(s,t)$ is a solution of the Wahlquist--Estabrook equation \eqref{EW}, then for every $t\in I$ fixed the function $f_t(s)$ is a solution of the Riccati equation \eqref{Riccati}.
\end{Remark}

\subsection{Geometric realization}

In order to describe the geometric realization of the B\"{a}cklund transformation for the KdV equation, we begin by recalling the definition of the LIEN flow and a result shown in~\cite{MP} regarding the procedure to construct the solutions of this flow.

Consider a smooth one parameter family of null curves $\gamma\colon (s,t)\in J\times I\subseteq\R^2\longmapsto \gamma(s,t)=\gamma_t(s)\in\AdS$ without inflection points and parameterized by the proper time. In other words, for each $t\in I$, we have a null curve $\gamma_t:J\subseteq\R\longrightarrow\AdS$ satisfying the assumptions of Remark~\ref{assumptions}.

The LIEN flow is the evolution equation for null curves in $\AdS$ given by
\begin{equation}\label{LIEN2}
	\partial_t\gamma=2\partial_s^3\gamma-6\bending\partial_s\gamma.
\end{equation}
In~\cite{MP}, we showed that the induced evolution equation on the bending $\bending$ of $\gamma$ is the KdV equation~\eqref{KdV2}. In addition, we gave a procedure to construct solutions of the LIEN flow \eqref{LIEN2} beginning with solutions $\bending$ of \eqref{KdV2} and employing suitable extended frames $E_\lambda$ of $\bending$. For the sake of completeness, we state this result here.

\begin{Theorem}[{\cite[Theorem 4.2]{MP}}]\label{induced}
	Let $\gamma\colon J\times I\subseteq\R^2\longrightarrow\AdS$ be a solution of the LIEN flow~\eqref{LIEN2}. Then, the bending $\bending(s,t)$ of $\gamma(s,t)$ evolves according to the KdV equation \eqref{KdV2}.
	
	Conversely, if $\bending\colon J\times I\subseteq\R^2\longrightarrow\R$ is a smooth solution of the KdV equation \eqref{KdV2}, then
	\begin{equation*}%\label{gamma}
		\gamma=E_1E_{-1}^{-1}\colon\ J\times I\subseteq\R^2\longrightarrow\AdS
	\end{equation*}
	is a solution of the LIEN flow \eqref{LIEN2} with bending $\bending$, where $E_\lambda$, $\lambda=-1,1$, are the extended frames of $\bending$ with spectral parameters $\lambda=-1,1$, respectively.
\end{Theorem}

\begin{Remark}\label{new}
	In addition to the statement of Theorem~\ref{induced}, in the proof of this result given in~\cite{MP}, it was also shown that the map $(E_1,E_{-1})$ is the spinor frame field along $\gamma_t$, for every fixed~${t\in I}$.
\end{Remark}

Combining this result and Theorem~\ref{T2} involving $\mathcal{T}$-transforms, we can give a geometric interpretation of the B\"{a}cklund transformation of the KdV as induced by the $\mathcal{T}$-transform of a~solution of the LIEN flow \eqref{LIEN2}.

\begin{Theorem}\label{T4} Let $\gamma\colon J\times I\subseteq\R^2\longrightarrow\AdS$ be a solution of the LIEN flow \eqref{LIEN2} with bending~$\bending(s,t)$ and denote by $(F_+,F_-)$ the spinor frame field along $\gamma(s,t)$. Given a constant $\xi\neq0$ and a solution $f\colon J\times I\subseteq\R^2\longrightarrow\R$ of the Wahlquist--Estabrook equation \eqref{EW} for $\lambda=\cosh(2\xi)$, define the map $\widetilde{\gamma}\colon J\times I\subseteq\R^2\longrightarrow\AdS$ by
	\begin{equation}\label{TLIEN}
		\widetilde{\gamma}=\pm F_+\begin{pmatrix} \tanh(\xi) & -\csch(\xi)\sech(\xi)f \\ 0 & \coth(\xi) \end{pmatrix} F_-^{-1}.
	\end{equation}
	Then,
	\begin{enumerate}\itemsep=0pt
		\item[$1.$] For every $t\in I$, the null curve $\widetilde{\gamma}_t$ is a $\mathcal{T}$-transform of $\gamma_t$ with parameter $\xi$ and transforming function $f_t$.
		\item[$2.$] The map $\widetilde{\gamma}\colon J\times I\subseteq\R^2\longrightarrow\AdS$ is a solution of the LIEN flow \eqref{LIEN2} where its bending $\widetilde{\bending}$ is the B\"{a}cklund transform of $\bending$ with spectral parameter $\lambda$ and transforming function $f$.
	\end{enumerate}
\end{Theorem}
\begin{proof} Suppose that $\gamma$ is a solution of the LIEN flow \eqref{LIEN2} and that $f$ satisfies the Wahlquist--Estabrook equation \eqref{EW}. As noticed in Remark~\ref{WE-R}, for every $t\in I$, the function $f_t(s)$ is a~solution of the Riccati equation \eqref{Riccati}.
	
	It then follows from Theorem~\ref{T2} that, for every $t\in I$ fixed, the null curve $\widetilde{\gamma}_t(s)=\widetilde{\gamma}(s,t)$ defined on \eqref{TLIEN} of the statement is a $\mathcal{T}$-transform of the null curve $\gamma_t(s)=\gamma(s,t)$ with parameter~$\xi$ and transforming function $f_t$. In addition, the bending $\widetilde{\bending}_t$ of $\widetilde{\gamma}_t$ is
	\begin{equation*}\widetilde{\bending}_t=-\bending_t+2f_t^2-2\cosh(2\xi),\end{equation*}
	where $\bending_t$ is the bending of $\gamma_t$. Clearly, since $\lambda=\cosh(2\xi)$, $\widetilde{\bending}_t$ is the B\"{a}cklund transform of $\bending_t$ with spectral parameter $\lambda$ and transforming function $f_t$.
	
	As mentioned above, the B\"{a}cklund transform $\widetilde{\bending}_t$ is also a solution of the KdV equation \eqref{KdV2}. Hence, we deduce from Theorem~\ref{induced} (see also Remark~\ref{new}) and the expressions of $\widetilde{F}_+$ and $\widetilde{F}_-$ given in \eqref{T1:f4} and \eqref{T1:f4-}, that $\widetilde{\gamma}\colon J\times I\subseteq\R^2\longrightarrow\AdS$ is a solution of the LIEN flow.
\end{proof}

\begin{Definition} In analogy with the $\mathcal{T}$-transform (for $\chi=0$) on null curves, the solution $\widetilde{\gamma}$ of the LIEN flow \eqref{LIEN2} given by \eqref{TLIEN} is called the \emph{$\mathcal{T}$-transform of $\gamma$ with parameter $\xi$ and transforming function $f$}. With some abuse of notation, we also denote it by $\mathcal{T}_{\xi,f}(\gamma)$.
\end{Definition}

The extension to solutions of the LIEN flow \eqref{LIEN2} of the $\mathcal{T}$-transform for $\chi=0$ on null curves also satisfies a permutability theorem.

\begin{Theorem}\label{5.13} Let $\gamma\colon J\times I\subseteq\R^2\longrightarrow\AdS$ be a solution of the LIEN flow \eqref{LIEN2} with bending $\bending(s,t)$ and consider two $\mathcal{T}$-transforms of $\gamma$, namely, $\gamma_1=\mathcal{T}_{\xi_1,f_1}(\gamma)$ and $\gamma_2=\mathcal{T}_{\xi_2,f_2}(\gamma)$ with parameters $\xi_1\neq \xi_2$ and transforming functions $f_1\neq f_2$, respectively. Then, the function\footnote{Observe that the functions $f_{2\to 1}$, $f_{1\to 2}$, $\bending_{2\to 1}$ and $\bending_{1\to 2}$ of the statement of Theorem~\ref{5.13} are functions of two variables $(s,t)\in J\times I\subseteq\R^2$, while the analogue ones of Theorem~\ref{T3} are functions of just one variable $s\in J\subseteq\R$. For simplicity, we avoid explicitly describing this in the statement.} $f_{2\to 1}$ $($defined as in Theorem $\ref{T3})$ satisfies the Wahlquist--Estabrook equation \eqref{EW} for $\bending_1$ and $\lambda_2=\cosh(2\xi_2)$, while $f_{1\to 2}$ $($also defined as in Theorem $\ref{T3})$ satisfies \eqref{EW} for $\bending_2$ and $\lambda_1=\cosh(2\xi_1)$. Moreover,
$\mathcal{T}_{\xi_2,f_{2\to 1}}(\mathcal{T}_{\xi_1,f_1}(\gamma))=\mathcal{T}_{\xi_1,f_{1\to 2}}(\mathcal{T}_{\xi_2,f_2}(\gamma))$,
	and the bending of $\mathcal{T}_{\xi_2,f_{2\to 1}}(\mathcal{T}_{\xi_1,f_1}(\gamma))$ is the function $\bending_{2\to 1}=\bending_{1\to 2}$ $($defined as in Theorem $\ref{T3})$. In particular, the bending $\bending_{2\to 1}$ is a solution of the KdV equation \eqref{KdV2}.
\end{Theorem}
\begin{proof}
	The proof is analogous to that of Theorem~\ref{T3}. The only remaining part to prove is that $f_{2\to 1}$ and $f_{1\to 2}$ satisfy the suitable Wahlquist--Estabrook equations \eqref{EW}. However, this is a consequence of the classical permutability theorem for the B\"{a}cklund transform of the KdV equation \eqref{KdV2}.
\end{proof}

\section{Construction procedure}\label{section5}

We finish this paper illustrating how to implement the construction of the $\mathcal{T}$-transforms for $\chi=0$ for solutions of the LIEN flow \eqref{LIEN2} with constant bending.

Consider a constant solution $\bending$ of the KdV equation \eqref{KdV2}. We first build the extended frames~$E_\lambda$ of $\bending$ with spectral parameter $\lambda\in\R$. To this end, we must integrate the $\mathfrak{sl}(2,\mathbb{R})$-valued $1$-form $\Gamma_\lambda$ given in \eqref{alpha}.
Since $\bending$ is constant, it follows from the Maurer--Cartan compatibility equation ${\rm d}\Gamma_\lambda+\Gamma_\lambda\wedge\Gamma_\lambda=0$ that $\mathcal{K}_\lambda$ and $\mathcal{P}_\lambda$ commute (see \eqref{KP} for their definition). Therefore, the map $E_\lambda(s,t)={\rm Exp}(\mathcal{K}_\lambda s+\mathcal{P}_\lambda t)$ satisfies \eqref{dE} and so it is an extended frame of $\bending$ with spectral parameter $\lambda\in\mathbb{R}$. A computation involving the definition of $\mathcal{K}_\lambda$ and $\mathcal{P}_\lambda$ given in \eqref{KP} shows that
\begin{equation*}
	E_{\lambda}(s,t) = \begin{pmatrix}\cosh\sigma(s,t)&\sqrt{\bending+\lambda}\,\sinh\sigma(s,t)\\
		\frac{1}{\sqrt{\bending+\lambda}}\,\sinh\sigma(s,t)&\cosh\sigma(s,t)\end{pmatrix},
\end{equation*}
where $\sigma(s,t)=\sqrt{\bending+\lambda}(s+2[\bending-2\lambda]t)$.
Here, we are understanding that $\sinh({\rm i}x)={\rm i}\sin(x)$ and $\cosh({\rm i}x)=\cos(x)$.

We will next construct a B\"{a}cklund transform of $\bending$ with spectral parameter $\lambda\neq 0$ and transforming function $f$. Recall that the transforming function is a solution of the Wahlquist--Estabrook equation \eqref{EW}. Hence, we will use the method explained in Subsection~\ref{section4.1} to construct this solution. If $c\in\R$ is the initial condition, from \eqref{TF} we deduce that the solution $f=-x/y$ of the Wahlquist--Estabrook equation \eqref{EW} is
\begin{equation}\label{f}
	f(s,t)=\frac{c\sqrt{\bending+\lambda}\,\cosh\sigma(s,t)+(\bending+\lambda)\sinh\sigma(s,t)}{\sqrt{\bending+\lambda}\,\cosh\sigma(s,t)+c\,\sinh\sigma(s,t)}.
\end{equation}
Then, $\widetilde{\bending}=-\bending -2\lambda +2f^2$
is the B\"{a}cklund transform  of  $\bending$ with spectral parameter $\lambda\in\R$, that is, a traveling wave solution of the KdV equation \eqref{KdV2}. These solutions are also known as $1$-soliton solutions.

The extended frame $\widetilde{E}_{\omega}$ of $\widetilde{\bending}$ with spectral parameter $\omega\neq \lambda$ is given by
\begin{equation}\label{recon}
	\widetilde{E}_{\omega}(s,t)=\frac{1}{\omega-\lambda}R(-f,\lambda,\omega) E_{\lambda}(s,t) R(-f,\lambda,\omega),
\end{equation}
where
\begin{equation*}R(x,y,z)=\begin{pmatrix}x&x^2-y+z\\1&x\end{pmatrix}.\end{equation*}
This can be verified by checking that $\widetilde{E}_{\omega}$ is a solution of \eqref{dE} for $\widetilde{\bending}$ and $\omega\in\R$.

Iterating the process and using again \eqref{TF} to construct a solution of the Wahlquist--Estabrook equation \eqref{EW} as above \big(but now for $\widetilde{E}_\omega$ given in \eqref{recon}\big), one can compute the transforming function\footnote{After the first step, the construction of the extended frames involve only algebraic manipulations, which can be performed with the help of any software of symbolic computations. However, the resulting formulas are very long and, hence, they have been omitted here.} \smash{$\widetilde{f}$} of \smash{$\widetilde{\bending}$}  with spectral parameter $\omega\neq \lambda$ and initial condition $\widetilde{c}$ as well as the corresponding B\"{a}cklund transform $\widehat{\bending}$ of $\widetilde{\bending}$, which is a $2$-soliton solution of the KdV equation \eqref{KdV2}.

\begin{Remark} At each step, the transforming functions may have singularities. Therefore, to obtain regular solutions, the constants involved in the construction must be chosen appropriately.
\end{Remark}

In the following example, we consider the particular case where the original null curve is closed and has constant bending.

\begin{Example}\label{5.14} Consider the constant solutions of the KdV equation \eqref{KdV2} given by (cf.\ Example~\ref{example})%
	\begin{equation*}\bending_{m,n}=-\frac{m^2+n^2}{m^2-n^2},\end{equation*}
	where $m>n$ are relatively prime natural numbers. In these cases, the associated null curves~$\gamma_{m,n}$ in $\AdS$ are closed.
	
	Then, for a real number $p>0$ fixed we can construct the transforming function $f_{m,n}$ of $\bending_{m,n}$ with spectral parameter
	\begin{equation*}%\label{lambda}
		\lambda_p = p + \frac{m^2+n^2}{m^2-n^2},
	\end{equation*}
	and initial condition $c=0$, simply substituting this data in \eqref{f}. Moreover, from \eqref{backlund}, we can also obtain the corresponding B\"{a}cklund transform $\widetilde{\bending}_{m,n}$ of $\bending_{m,n}$. In Figure~\ref{example1} we show the transforming function $f_{m,n}$ (left) and the corresponding B\"{a}cklund transform $\widetilde{\bending}_{m,n}$ (right) for suitable choices of $m>n$ and $p>0$.
\begin{figure}[t]\centering
				\includegraphics[height=4cm,width=4cm]{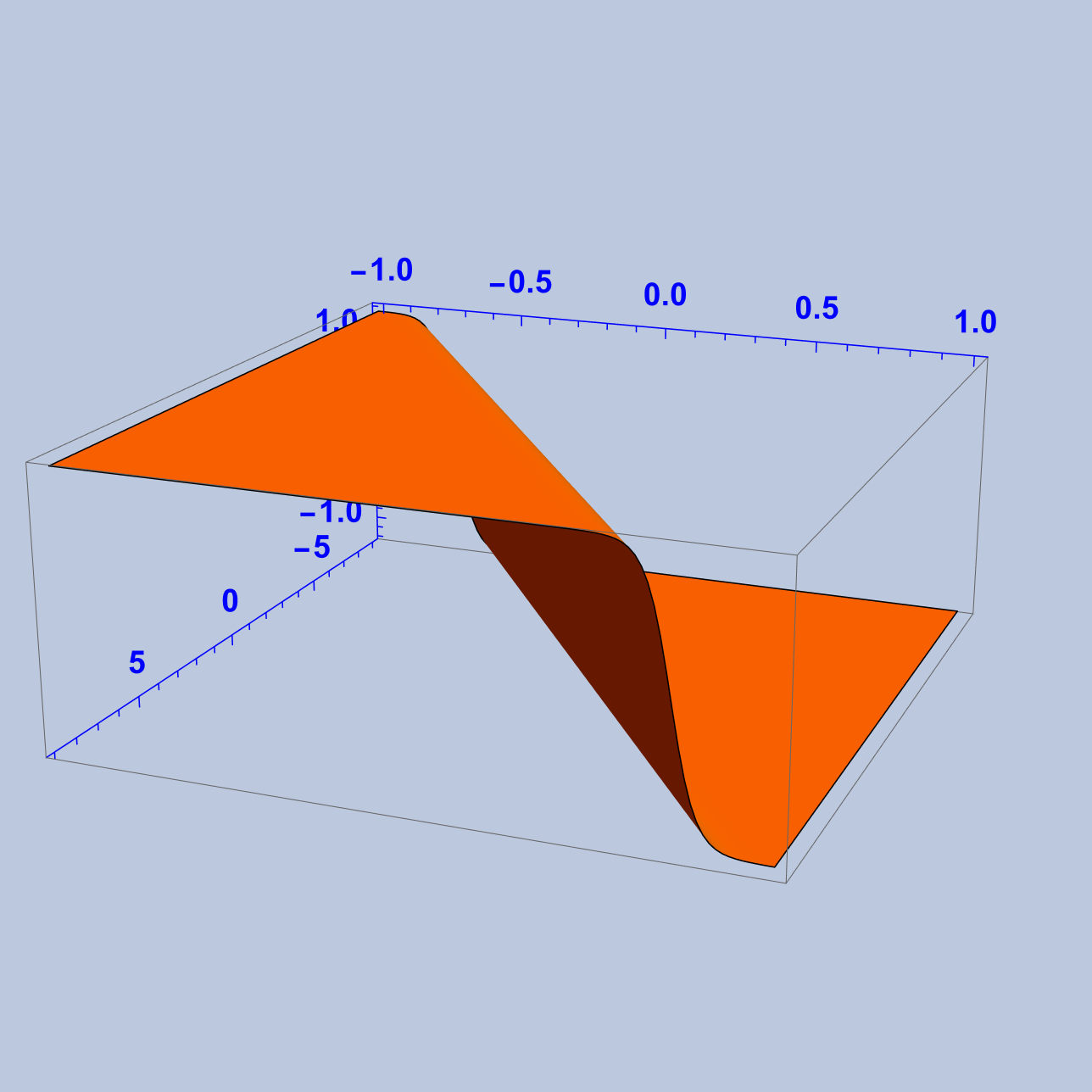}\quad\quad\quad
				\includegraphics[height=4cm,width=4cm]{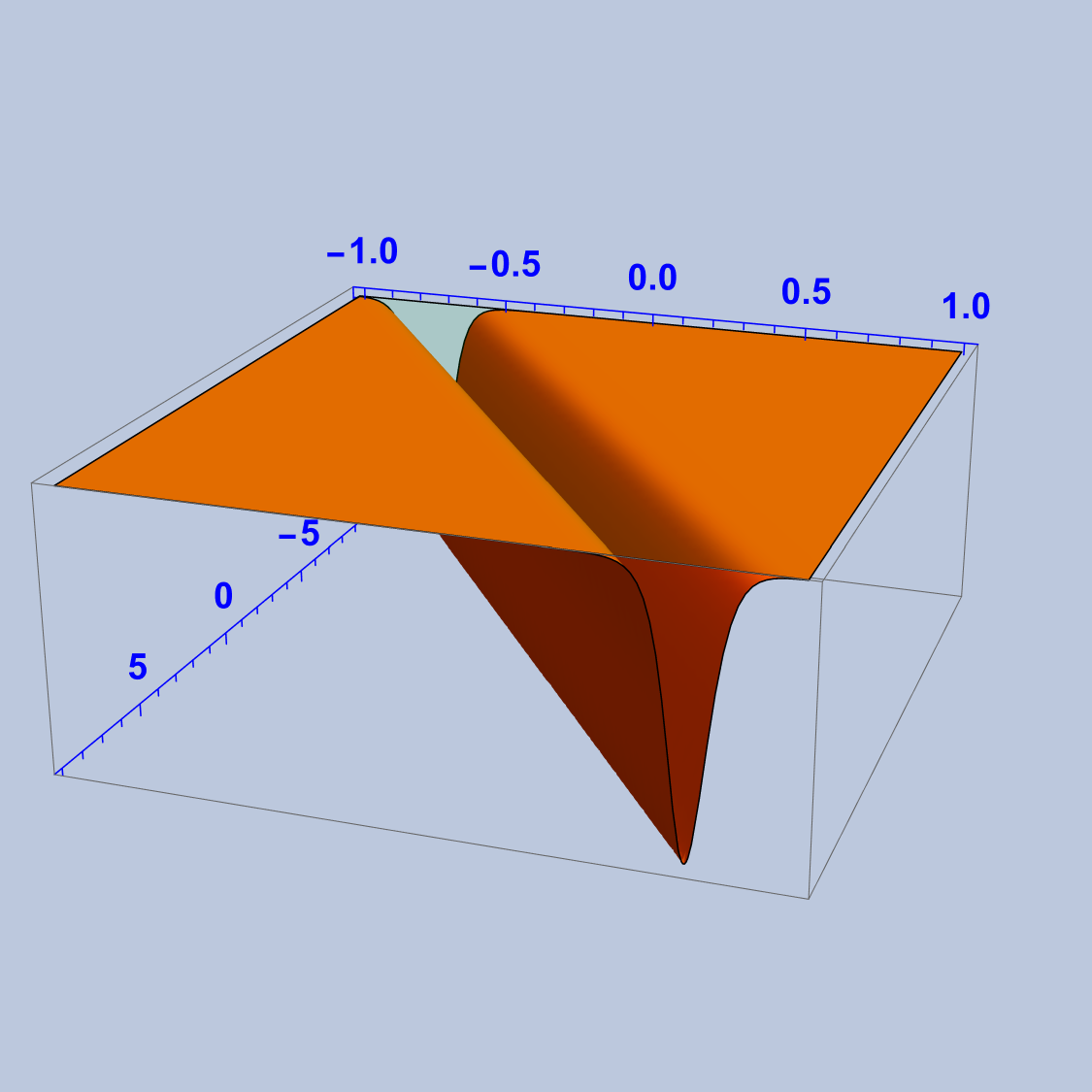}
\caption{Left: The transforming function $f_{m,n}$ of $\bending_{m,n}$ with spectral parameter $\lambda_p$ and initial condition~${c=0}$. Right: The B\"{a}cklund transform $\widetilde{\bending}_{m,n}$ of $\bending_{m,n}$ with spectral parameter $\lambda_p$ and transforming function~$f_{m,n}$. The function $\widetilde{\bending}_{m,n}$ represents a $1$-soliton solution of the KdV equation \eqref{KdV2}. In both cases, $m=4$, $n=1$ and $p=1.4$.} \label{example1}
	\end{figure}
	
	Employing \eqref{TF} and \eqref{recon}, we may iterate the process obtaining the transforming function~$\widetilde{f}_{m,n}$ of~$\widetilde{\bending}_{m,n}$ with spectral parameter
	\begin{equation*}\omega_{p,r}=p+r+\frac{m^2+n^2}{m^2-n^2},\qquad r>0,\end{equation*}
	and initial condition $\widetilde{c}=0$, as well as the B\"{a}cklund transform \smash{$\widehat{\bending}_{m,n}$} of \smash{$\widetilde{\bending}_{m,n}$} with spectral parameter $\omega_{p,r}$ and transforming function \smash{$\widetilde{f}_{m,n}$}. An example of the transforming function \smash{$\widetilde{f}_{m,n}$} and of its corresponding B\"{a}cklund transform is illustrated in Figure~\ref{example2} (left and right, respectively).
	
	\begin{figure}[t]
\centering
				\includegraphics[height=4cm,width=4cm]{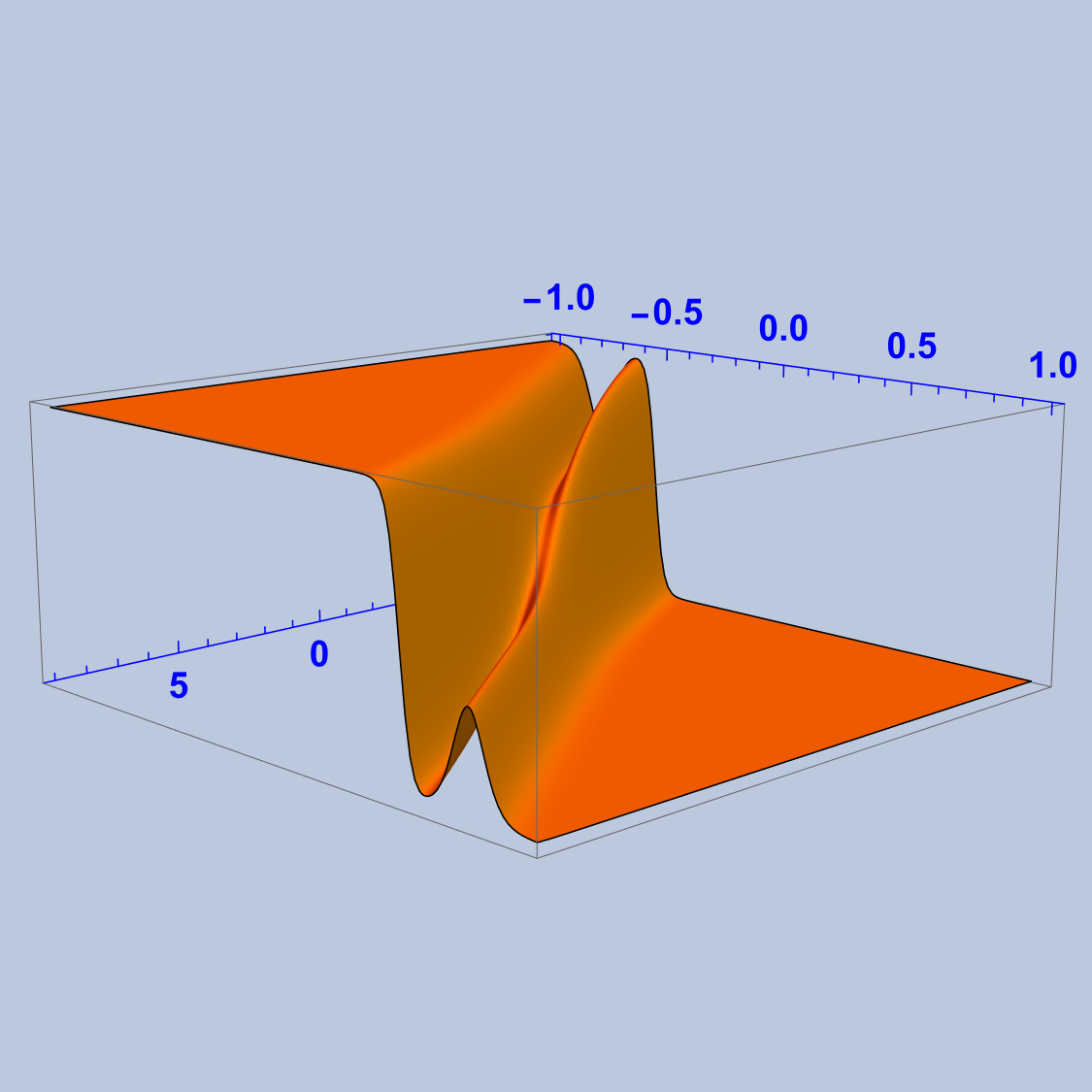}\quad\quad\quad
				\includegraphics[height=4cm,width=4cm]{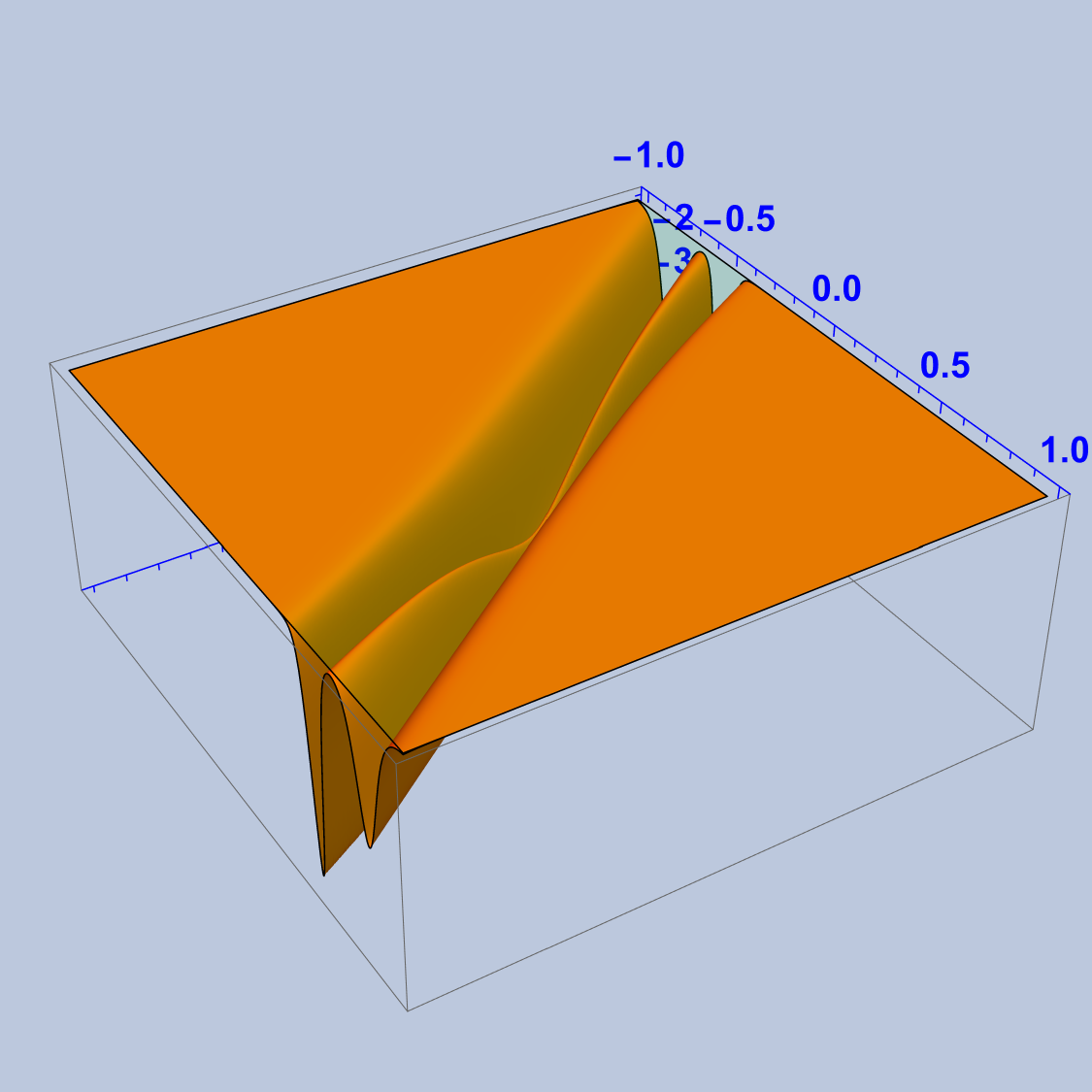}
			\caption{Left: The transforming function $\widetilde{f}_{m,n}$ of $\widetilde{\bending}_{m,n}$ with spectral parameter $\omega_{p,r}$. Right: The~B\"{a}ck\-lund transform $\widehat{\bending}_{m,n}$ of $\widetilde{\bending}_{m,n}$ with spectral parameter $\omega_{p,r}$ and transforming function $\widetilde{f}_{m,n}$. The~function~$\widehat{\bending}_{m,n}$ represents a $2$-soliton solution of the KdV equation \eqref{KdV2}. In both cases, $m=4$, $n=1$, $p=1.4$ and $r=1$.} \label{example2}
	\end{figure}
	
	As explained in Example~\ref{example}, the null curve $\gamma_{m,n}$ in $\AdS$ with bending $\bending_{m,n}$ is a torus knot. In addition, we deduce from Theorem~\ref{induced} that its evolution by the LIEN flow \eqref{LIEN2} is given by $\gamma=E_1E_{-1}^{-1}$, where $E_\lambda$ are the extended frames of $\bending_{m,n}$ with spectral parameter $\lambda=-1,1$, respectively.
	
	Let $\xi_\lambda\neq 0$ be a constant such that $\cosh(2\xi_\lambda)=\lambda_p$. Then, the $\mathcal{T}$-transform for $\chi=0$ of~$\gamma$ with parameter $\xi_\lambda$ and transforming function $f_{m,n}$ is $\mathcal{T}_{\xi_\lambda,f_{m,n}}(\gamma)$ (its explicit expression is given in \eqref{TLIEN} of Theorem~\ref{T4}). Since $\widetilde{\bending}_{m,n}$ is a traveling wave solution of the KdV equation~\eqref{KdV2}, it follows that~$\mathcal{T}_{\xi_\lambda,f_{m,n}}(\gamma)$ is a solution of the LIEN flow \eqref{LIEN2} consisting of the evolution of the initial condition $\gamma_{m,n}$ by rigid motions. Furthermore, for every $t\in I$, the function ${s\in\R\longmapsto\widetilde{\bending}_{m,n}(s,t)-\bending_{m,n}}$ is smooth and rapidly decaying (see the picture on the left of Figure~\ref{example3}). This implies that, for every $t\in I$, the null curve $s\in\R\longmapsto\mathcal{T}_{\xi_\lambda,f_{m,n}}(\gamma(s,t))$ tends asymptotically to two closed null curves with constant bending $\bending_{m,n}$ as $s\to\pm\infty$. More precisely, in practical terms, this curve is made up of two disjoint, but congruent, torus knots connected by an arc (see Figure~\ref{example4}).
	
	\begin{figure}[t]
\centering
				\includegraphics[height=4cm,width=4cm]{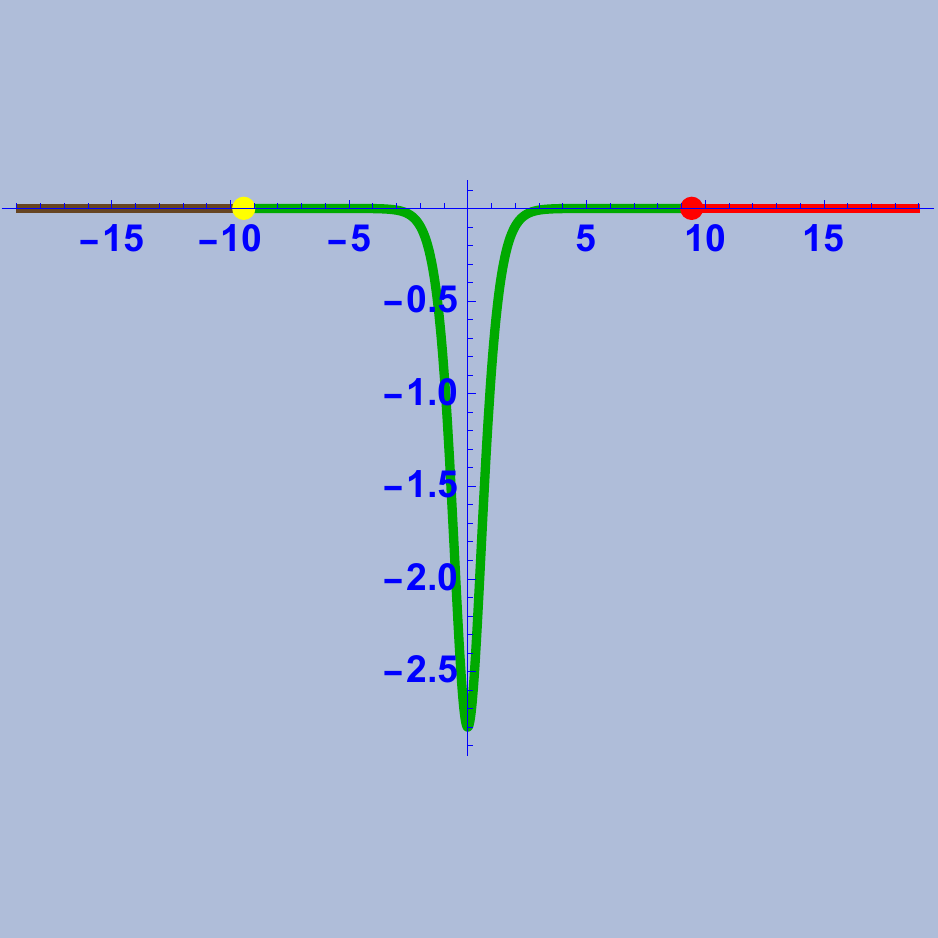}\quad\quad\quad
				\includegraphics[height=4cm,width=4cm]{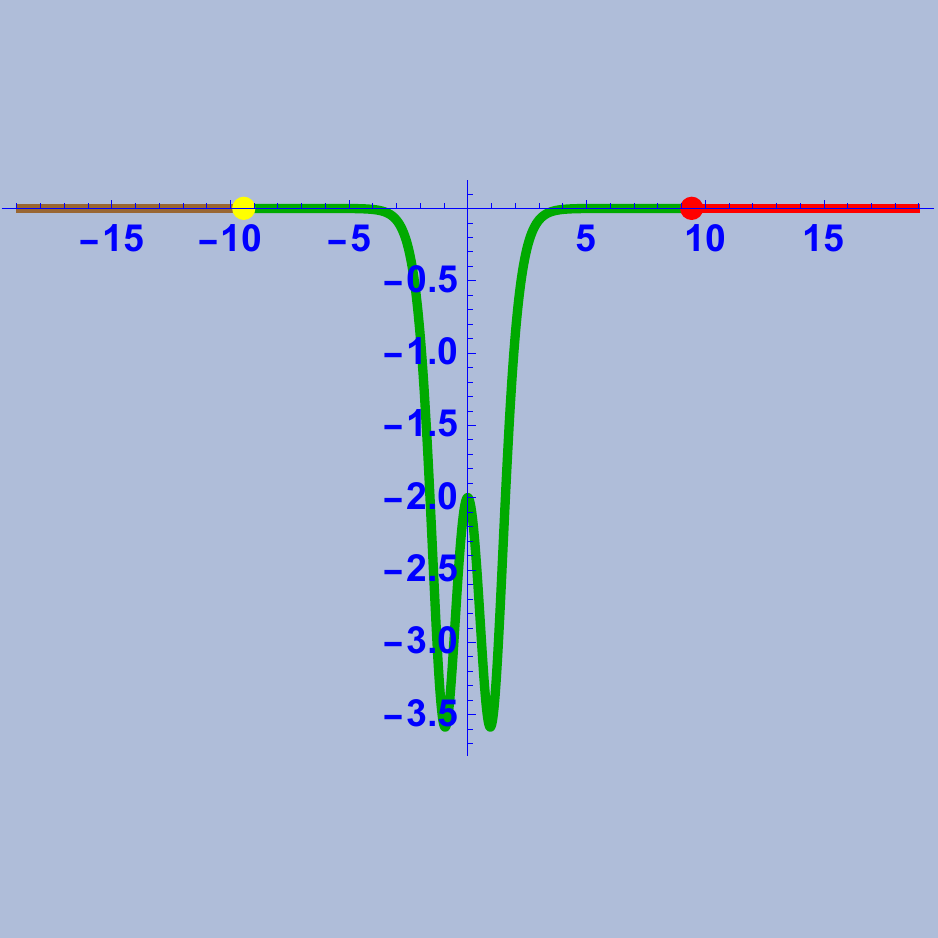}
			\caption{Left: The graph of the rapidly decaying function $s\in \R\longmapsto \widetilde{\bending}_{m,n}(s,0)-\bending_{m,n}$. On the left of the yellow point and on the right of the red point $\lvert\widetilde{\bending}_{m,n}(s,0)-\bending_{m,n}\rvert \lesssim 2.31\times 10^{-9}$.
					Right: The graph of the function $s\in\R\longmapsto  \widehat{\bending}_{m,n}(s,0)-\bending_{m,n}$. On the left of the yellow point and on the right of the red point $\lvert\widehat{\bending}_{m,n}(s,0)-\bending_{m,n}\rvert\lesssim 1.73\times 10^{-8}$. These graphs have been computed for the values
					$m=4$, $n=1$, $p=1.4$ and $r=1$.}  \label{example3}
	\end{figure}
	
	\begin{figure}[t]\centering
				\includegraphics[height=4cm,width=4cm]{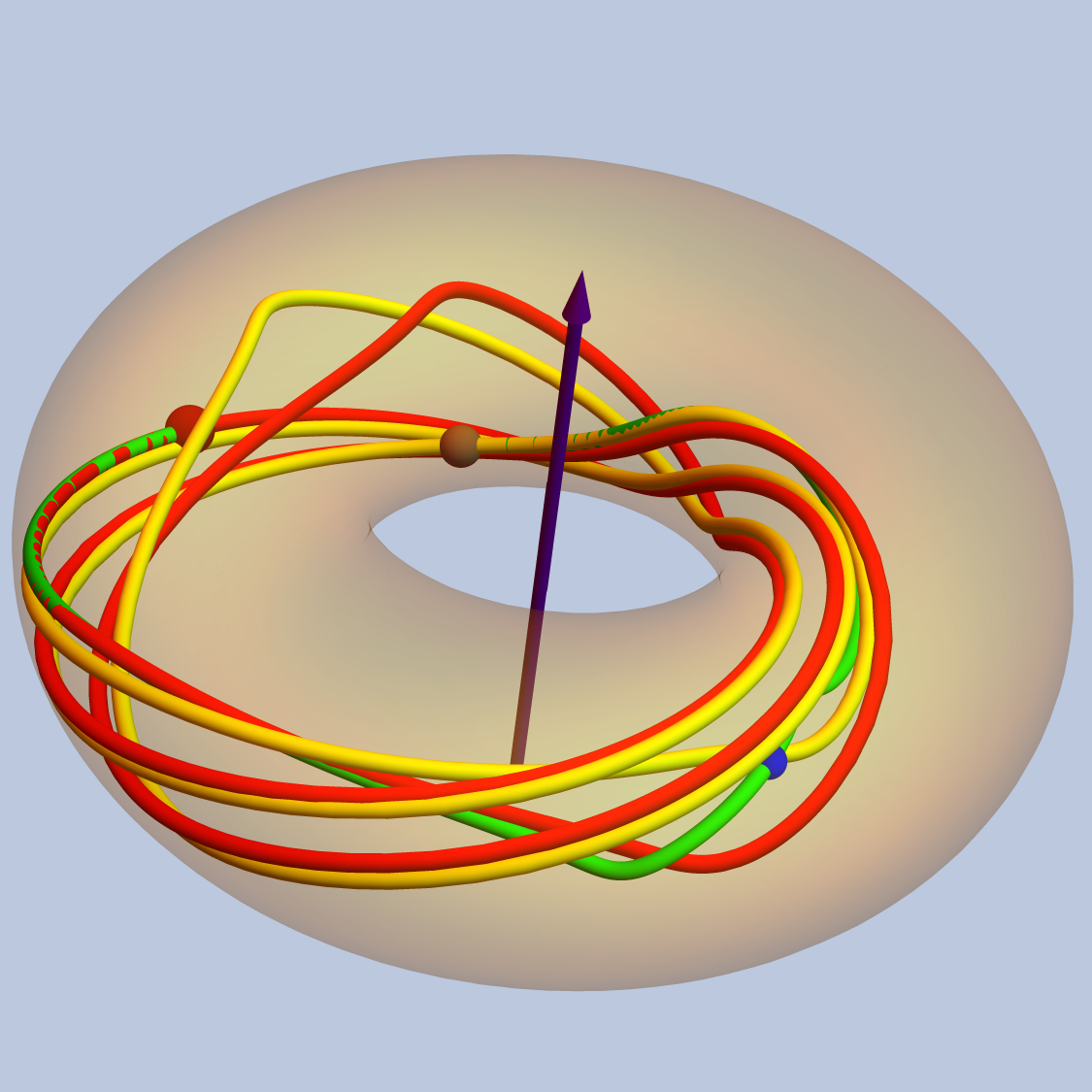}\quad
				\includegraphics[height=4cm,width=4cm]{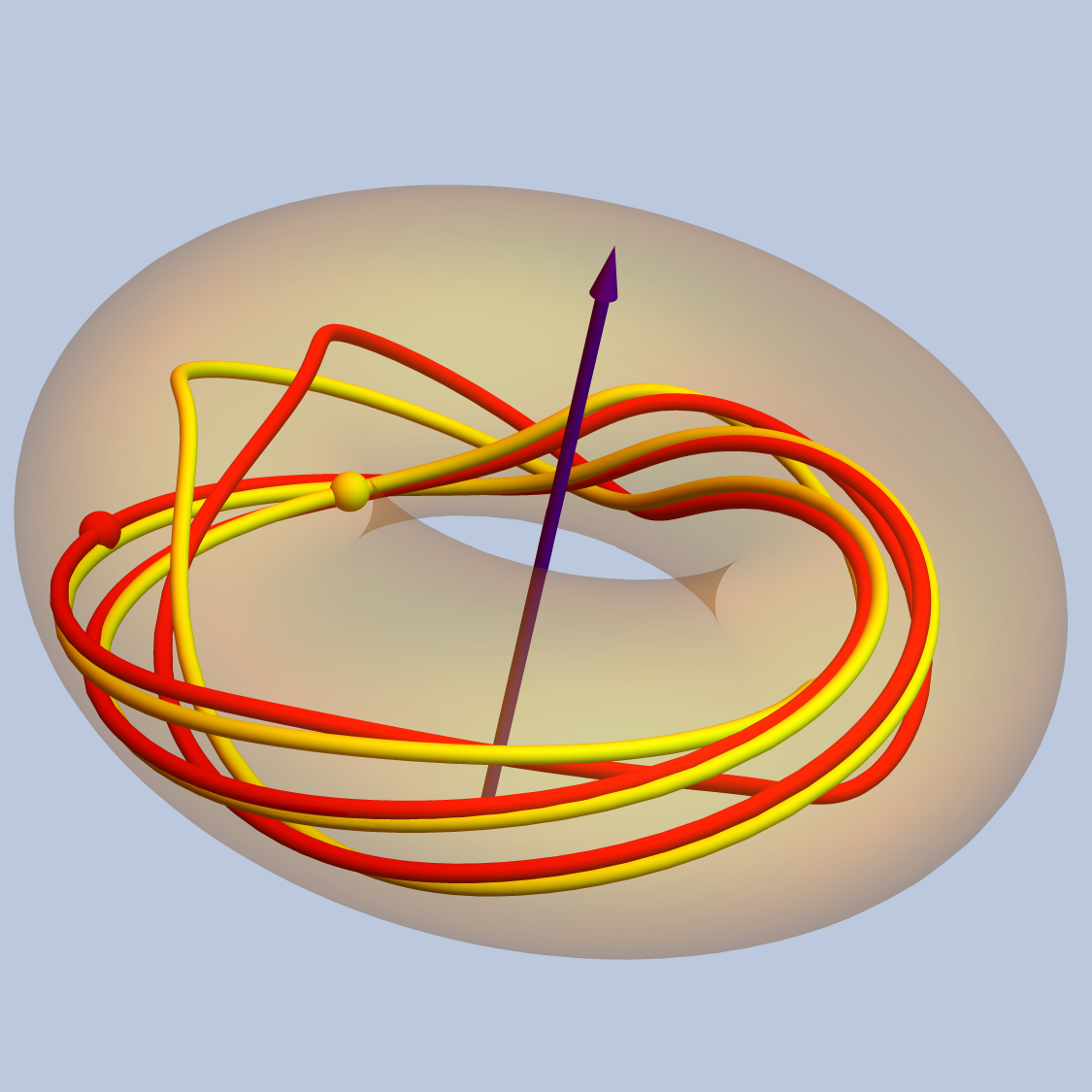}\quad
				\includegraphics[height=4cm,width=4cm]{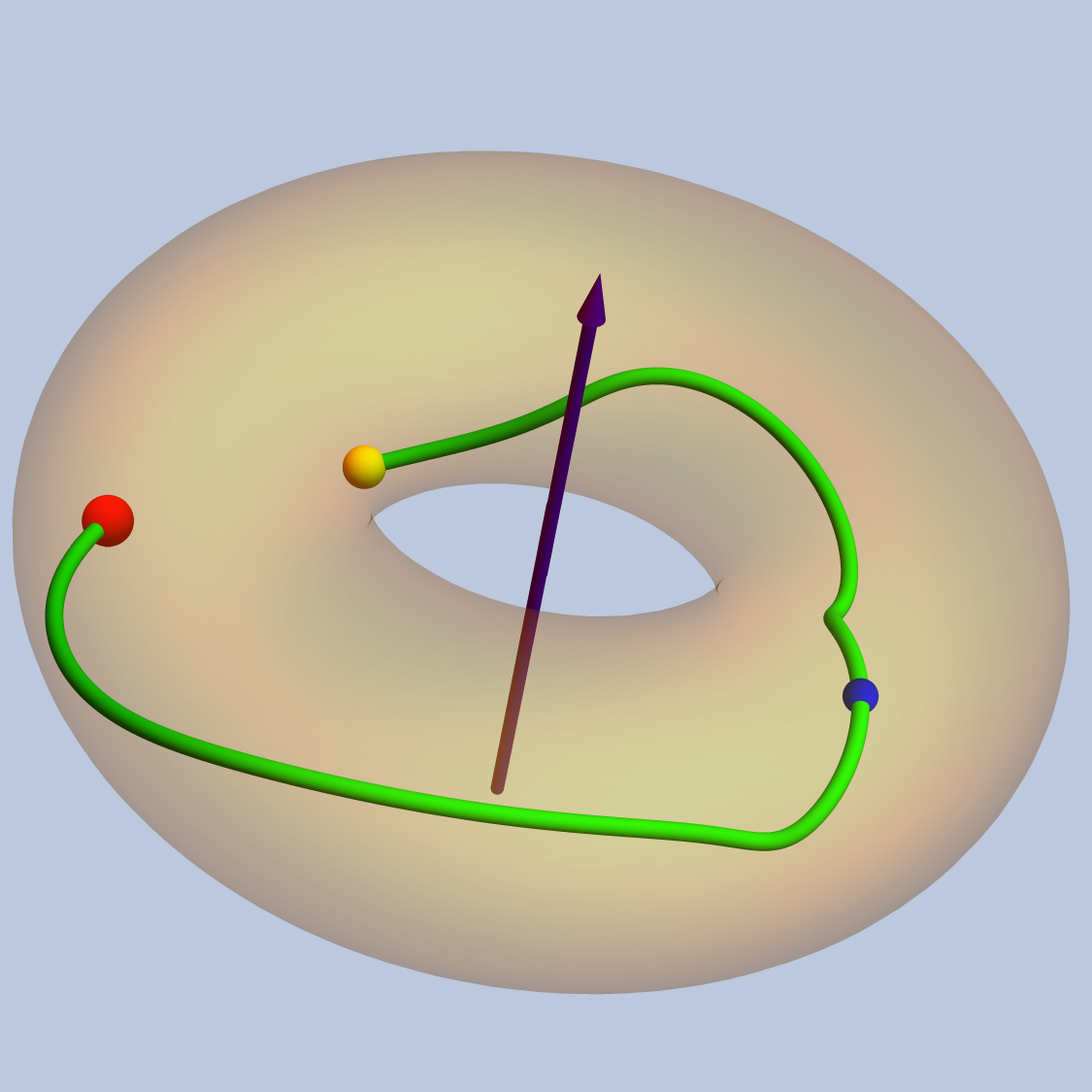}
			\caption{Left: The null curve $s\in\R \longmapsto {\mathcal T}_{\xi_{\lambda},f_{m,n}}(\gamma(s,0))$. Center: The two disjoint, but congruent torus knots. Right: The arc connecting the two torus knots. For this example, we have fixed $m=4$, $n=1$ and $p=1.4$.}  \label{example4}
	\end{figure}
	
	Finally, we build $\mathcal{T}$-transforms for $\chi=0$ of $\mathcal{T}_{\xi_\lambda,f_{m,n}}(\gamma)$. For this, we fix $r>0$ and a constant~${\xi_\omega\neq 0}$ such that $\cosh(2\xi_\omega)=\omega_{p,r}$. Using (\ref{recon}), we compute the extended frames $\widetilde{E}_\omega$ of $\widetilde{\bending}_{m,n}$ and the transforming function $\widetilde{f}_{m,n}$. Then, the ${\mathcal T}$-transform for $\chi=0$ of ${\mathcal T}_{\xi_{\lambda},f_{m,n}}(\gamma)$ with parameter  $\xi_{\omega}\neq 0$ and transforming function \smash{$\widetilde{f}_{m,n}$} is given by \eqref{TLIEN} and denoted by \smash{$\mathcal{T}_{\xi_\omega,\widetilde{f}_{m,n}}\bigl(\mathcal{T}_{\xi_\lambda,f_{m,n}}(\gamma)\bigr)$}. Moreover, the bending of \smash{$\mathcal{T}_{\xi_\omega,\widetilde{f}_{m,n}}\bigl(\mathcal{T}_{\xi_\lambda,f_{m,n}}(\gamma)\bigr)$} is the $2$-soliton solution $\widehat{\bending}_{m,n}$ of the KdV equation~\eqref{KdV2}. In this case, the evolution is not by rigid motions. However, as in the previous case,  the functions $s\in \R\longmapsto \widehat{\bending}_{m,n}(s,t)-\bending_{m,n}$ are rapidly decaying (see the picture on the right of Figure~\ref{example3}). Thus, the geometrical structure of the evolving curves is similar to the previous case (see Figure~\ref{example5}).
	
	In principle, the procedure can be inductively repeated to construct the iterated ${\mathcal T}$-transforms of $\gamma_{m,n}$.
\end{Example}

\begin{figure}[t]\centering
			\includegraphics[height=4cm,width=4cm]{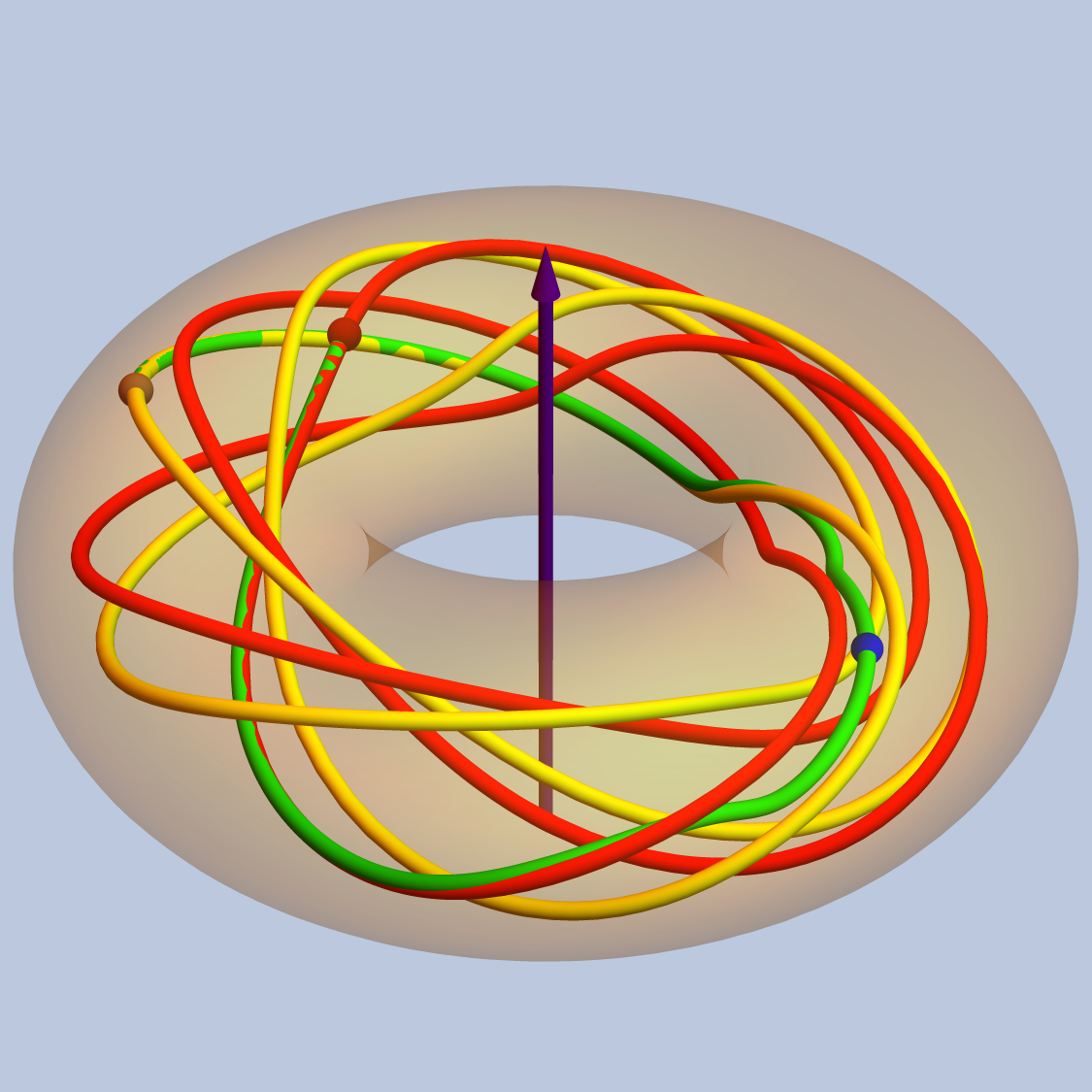}\quad
			\includegraphics[height=4cm,width=4cm]{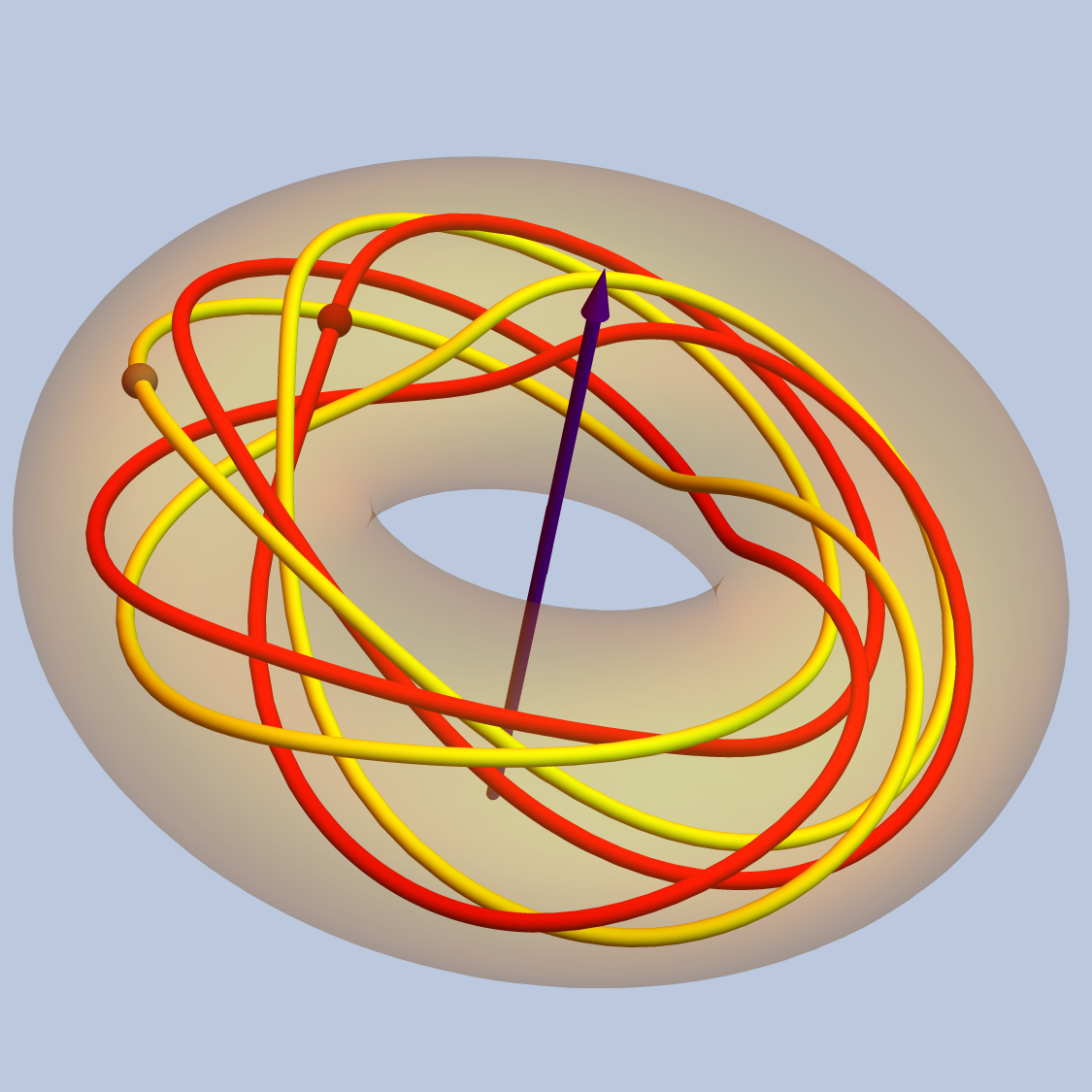}\quad
			\includegraphics[height=4cm,width=4cm]{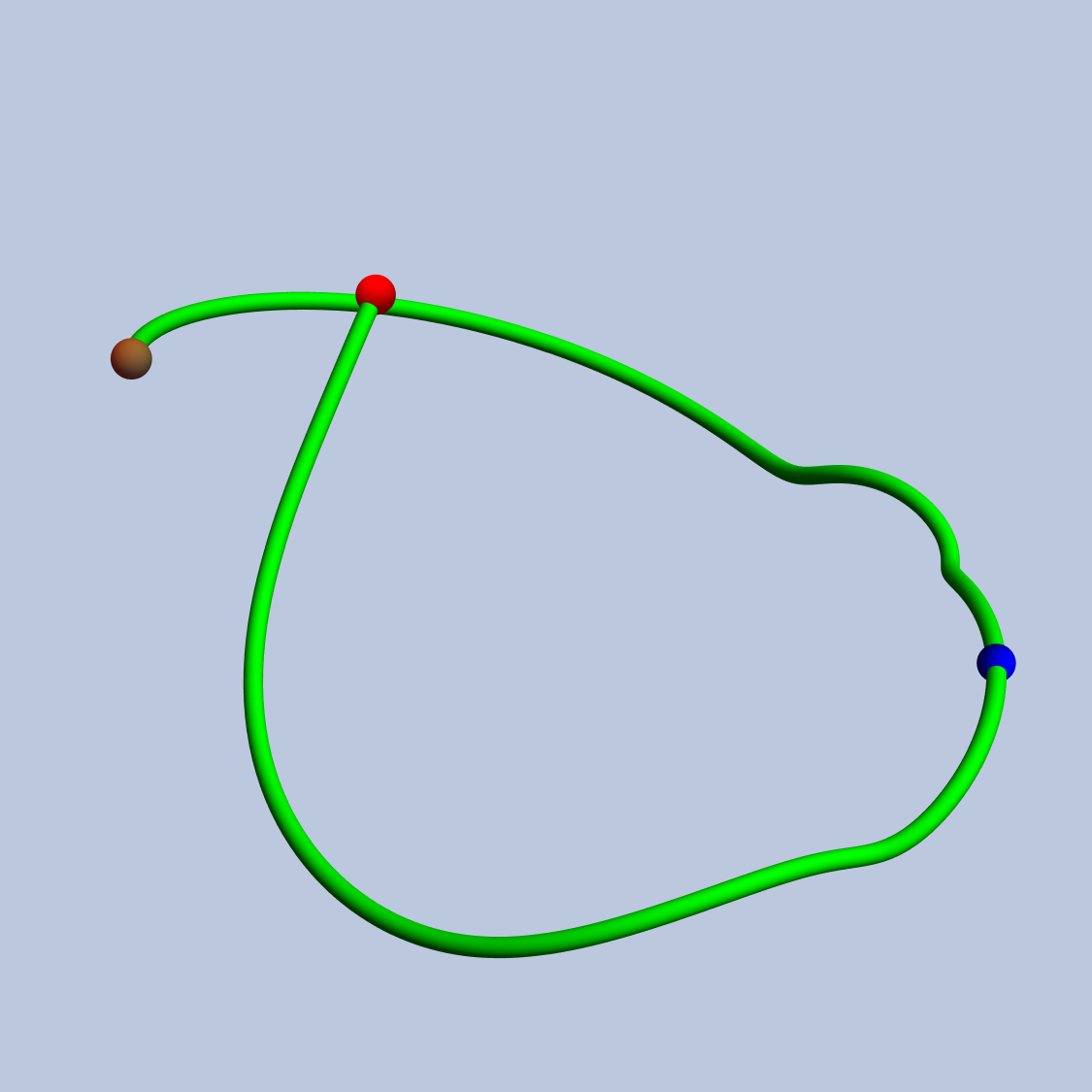}
		\caption{Left: The null curve $s\in\R \longmapsto \mathcal{T}_{\xi_\omega,\widetilde{f}_{m,n}}\bigl(\mathcal{T}_{\xi_\lambda,f_{m,n}}(\gamma(s,0))\bigr)$. Center: The two disjoint, but congruent torus knots. Right: The arc connecting the two torus knots. For this example, we have fixed $m=4$, $n=1$, $p=1.4$ and $r=1$.}  \label{example5}
\end{figure}

\subsection*{Acknowledgments}

Authors partially supported by PRIN 2017 and PRIN 2022 ``Real and Complex Manifolds: Topology, Geometry and Holomorphic Dynamics'' (prott. 2017JZ2SW5-004 and 2022AP8HZ9-003); and by the GNSAGA of INdAM. The authors gratefully acknowledge the warm hospitality of the Department of Mathematics at the Politecnico di Torino and of the Texas Tech University Center in Sevilla.
The authors would like to thank the referees for carefully reviewing the paper.

\pdfbookmark[1]{References}{ref}
\LastPageEnding

\end{document}